\documentclass{article}
 
\usepackage{amsmath} 
\usepackage{amssymb} 
\usepackage{amsthm} 
\usepackage{bbm}
\usepackage[utf8]{inputenc}
\usepackage{mathtools} 
\usepackage{authblk}
\usepackage{geometry}
\theoremstyle{plain}
\newtheorem{theorem}{Theorem}
\newtheorem{corollary}[theorem]{Corollary}
\newtheorem{definition}[theorem]{Definition}
\newtheorem{lemma}[theorem]{Lemma}
\newtheorem{proposition}[theorem]{Proposition}

\newtheorem{assumption}[theorem]{Assumption}
 
\newcommand{\dd}{\, \mathrm{d}}

\newcommand{\R}{\mathbb{R}}

\DeclareFontEncoding{FMS}{}{}
\DeclareFontSubstitution{FMS}{futm}{m}{n}
\DeclareFontEncoding{FMX}{}{}
\DeclareFontSubstitution{FMX}{futm}{m}{n}
\DeclareSymbolFont{fouriersymbols}{FMS}{futm}{m}{n}
\DeclareSymbolFont{fourierlargesymbols}{FMX}{futm}{m}{n}
\DeclareMathDelimiter{\VERT}{\mathord}{fouriersymbols}{152}{fourierlargesymbols}{147}

\title{Optimal control of a nonlinear kinetic Fokker-Planck equation}

\author[1]{Tobias Breiten}
\author[2]{Karl Kunisch}
\affil[1]{Institute of Mathematics, Technische Universit\"at Berlin, Stra\ss e des 17. Juni 136, 10623 Berlin, Germany,   \texttt{tobias.breiten@tu-berlin.de}}
\affil[2]{Institute of Mathematics, University of Graz, Austria and Radon Institute, Austrian Academy of Sciences, Linz, Austria,  \texttt{karl.kunisch@uni-graz.at}} 
\begin{document}

\maketitle

\begin{abstract}
   A tracking type optimal control problem for a nonlinear and nonlocal kinetic Fokker-Planck equation which arises as the mean field limit of an interacting particle systems that is subject to distance dependent random fluctuations is studied. As the equation of interest is only hypocoercive and the control operator is unbounded with respect to the canonical state space, classical variational solution techniques cannot be utilized directly. Instead, the concept of admissible control operators is employed. For the underlying nonlinearities, local Lipschitz estimates are derived and subsequently used within a fixed point argument to obtain local existence of solutions. Again, due to hypocoercivity, existence of optimal controls requires non standard techniques as (compensated) compactness arguments are not readily available.
\end{abstract}
 
{\bf Keywords:} hypocoercivity, nonlinear kinetic Fokker-Planck, optimal control \\

{\bf MSC codes:} 35Q83, 35Q84, 47N70, 49J20

\section{Introduction}\label{sec:intro}
 
We consider the following nonlinear nonlocal controlled Fokker-Planck equation
\begin{equation}\label{eq:orig_nonl_kfp}
 \begin{aligned}
  \partial_tf + v  \cdot \nabla_xf + U * \rho_{v  f}\cdot \nabla _v  f &= U* \rho_f \nabla_v  \cdot (\nabla_v  f + v  f)+ u( \alpha \cdot \nabla_v f )\\
  f(0,x,v )&=f_0(x,v ),
 \end{aligned}
 \end{equation}
where $\alpha \in L^2(\mathbb R^d)\cap L^\infty(\mathbb R^d)$ only depends on $x\in \mathbb R^d$, $u\in L^\infty(0,T)$ is a scalar time-dependent control and
\begin{align*}
 \rho_f(t,x)=\int_{\mathbb R^d} f(t,x,v )\,\mathrm{d}v , \ \ \rho_{v  f}(t,x) = \int_{\mathbb R^d} v  f(t,x,v )\,\mathrm{d}v,
\end{align*}
where $\rho_f(t,x)\in\mathbb R$ and $\rho_{vf}(t,x)\in \mathbb R^d$. Throughout the manuscript, we shall also consider $\rho_f$ and $\rho_{vf}$ as time-independent functions, e.g., when $f$ is fixed for a specific time $t$. For $u\equiv 0$, equation \eqref{eq:orig_nonl_kfp} has been introduced in \cite{DuaFT10} and shown to arise as a mean-field limit of the (stochastically perturbed) particle system
\begin{equation}\label{eq:part_sys}
\begin{aligned}
  \dd x_i &= \dd v_i \dd t \\
  \dd v_i &= \sum_{i=1}^m U(\|x_j-x_i\|)(v_j-v_i) \dd t + \sqrt{2 \mu \sum_{j=1}^m U(\|x_j-x_i\|)} \dd W_i 
 \end{aligned}
\end{equation}
 where $(x_i,v_i) \in \mathbb R^d\times \mathbb R^d, i=1,\dots,m$ model the evolution of the position and velocity of a collection of particles, e.g., birds that communicate with each other via the distance potential or communication rate $U$. The model \eqref{eq:part_sys} is a generalization of the classical Cucker-Smale model from \cite{CucS07} and has been analyzed in \cite{HaT08} for a non-integrable distance potential $U$. In \cite{DuaFT10}, the noiseless model from \cite{HaT08} has been augmented by random fluctuations of the particles which depend on the strength of their underlying density. Here, we follow \cite{DuaFT10} and assume the potential $U=U(x)$ to be continuous and to satisfy
\begin{align*}
 U(x) = U(\|x\|) \ge 0, \ \ \int_{\mathbb R^d} U(x)\,\mathrm{d}x=1.
\end{align*}
It is well-known and can easily be verified that a steady state solution to the uncontrolled equation \eqref{eq:orig_nonl_kfp} is given by the \textit{Maxwellian} function
\begin{align}\label{eq:maxwellian}
  \mu=\mu(v) = (2\pi )^{-\tfrac{d}{2}}e^{-\frac{\|v\|^2}{2}}.
\end{align}
In fact, for initial values $f_0$ close to $\mu$, the unique solution to \eqref{eq:orig_nonl_kfp} with $u\equiv 0$ asymptotically converges algebraically to $\mu$ as $t\to \infty$, see \cite[Theorem 1.1]{DuaFT10}. 

Since its introduction in \cite{CucS07}, many works dealing with the analysis, control and numerical simulation of particle systems, and  with the PDEs describing their associated densities have appeared in the literature \cite{BelDT17}. 
In view of equation \eqref{eq:orig_nonl_kfp}, let us for example mention the early work \cite{CarS95} where the authors prove a global existence result for a nonlinear nonlocal Vlasov-Poisson-Fokker-Planck system which, similar to \eqref{eq:orig_nonl_kfp}, exhibits both parabolic (in the variable $v$) and hypberbolic (in the variable $x$) behavior, a property often found in so-called hypocoercive equations \cite{Vil09}. We also refer to \cite{Cho16,ChoHY23} where other nonlinear Fokker-Planck type equations have been analyzed w.r.t.~global existence and uniqueness of solutions as well as \cite{KarMT13} which, in contrast to \eqref{eq:orig_nonl_kfp}, however, deals with nonlinearities not entering the highest order differential operator. For the noiseless version of \eqref{eq:orig_nonl_kfp}, an optimal control problem has been studied in \cite{PicRT15}, and \cite{Albetal17} considers deterministic and non-deterministic control problems for mean-field PDEs depending only on the position variable $x$.

In addition to its nonlinear and nonlocal nature, the analysis of \eqref{eq:orig_nonl_kfp} is challenging due to the above mentioned hypocoercivity which does not allow for standard coercivity arguments when studying the long-time behavior of solutions. For a general overview on hypocoercive equations, let us refer to, e.g., the monographs \cite{HelN05,Vil09} as well as the survey article \cite{LelS16}. For linear (uncontrolled) Fokker-Planck equations, a detailed treatise of hypocoercivity results can be found for example in \cite{ConG08,GroS16} and, more recently, also in \cite{Albetal24}. In   \cite{AAS15} the long time behavior of linear hypocoercive Fokker-Planck equations is analyzed with a modified entropy method. Recently in a series of papers of which we  cite \cite{AAMN24} short and long term decay rates were analyzed using the hypocoercivity index. 
In \cite{BreK23}, we have considered an (infinite-horizon) control problem for a linear hypocoercive Fokker-Planck equation which bears resemblance to a linearization of \eqref{eq:orig_nonl_kfp} around the Maxwellian $\mu$, cf.~the subsequent perturbation discussion below. 

With regard to the above literature, our contribution is twofold. On the one hand, the introduction of a control interaction in \eqref{eq:part_sys} and thus also \eqref{eq:orig_nonl_kfp}, e.g., by an \emph{opinion leader} requires the analysis of a nonlinear nonlocal nonhomogeneous hypocoercive PDE that is not availabe elsewhere. Using the control theoretic concept of \emph{admissibility} in combination with a fixed point strategy (cf.~also \cite{HosJS18} for an abstract bilinearly controlled Cauchy problem), for small (close to the Maxwellian $\mu$) initial data $f_0$ and small control $u$, in Theorem \ref{thm:weak_non_con} we will obtain the local existence of a mild solution to \eqref{eq:orig_nonl_kfp}. In contrast to the existence result (for the uncontrolled solution) from \cite{DuaFT10}, we do not require smooth initial data but only assume $f_0$ square integrable on $\mathbb R^{2d}$ with respect to the canonical invariant measure characterized by $\mu$. On the other hand we cannot assert uniqueness for the resulting notion of weak solutions.  Our second main result concerns the existence of a locally optimal solution to a quadratic tracking type cost functional, see Theorem \ref{thm:ex_opt_cont}, and the uniqueness of the associated optimal state, cf.~Proposition \ref{prop:unique}.  

While we concentrate here on a concrete nonlinear kinetic Fokker-Planck equation, some of the concepts that we employ  could be useful for the study of other optimal control problems for different nonlinear kinetic equations as well. These include the combination of semigroup and variational techniques, the  use of admissible control operators, and the treatment of optimal control problems where the control to state mapping is not necessarily unique.

The structure of the manuscript is as follows.  For the analysis of \eqref{eq:orig_nonl_kfp} it will be convenient to consider variables evolving locally around the Maxwellian $\mu$. Therefore we next derive an equation, equivalent to \eqref{eq:orig_nonl_kfp}, for the perturbation variable $y$. This section ends with an introduction of the  notation used throughout the remainder of the manuscript. Section \ref{sec:lin_eq} analyzes a suitable linearization by means of semigroup as well as variational techniques similarly as utilized in \cite{BreK23}. The nonlinear equation is studied in section \ref{sec:nonlin_eq}. Here, Lipschitz estimates for the nonlinearities appearing in \eqref{eq:orig_nonl_kfp} are derived  and  the local existence of solutions by a fixed point argument is established. In section \ref{sec:opt_con_prob}, we discuss a quadratic tracking type cost functional for which we discuss existence of an optimal control as well as uniqueness of its associated state. The manuscript ends with a conclusion and an outlook of potential future research questions in the context of \eqref{eq:orig_nonl_kfp}.

\paragraph{A perturbed version of \eqref{eq:orig_nonl_kfp}.}

Subsequently, we will derive an equation equivalent to \eqref{eq:orig_nonl_kfp} by considering the dynamics in relation to the steady state $\mu(v)= (2\pi )^{-\tfrac{d}{2}}e^{-\frac{\|v\|^2}{2}}$. For this purpose as well as for several calculations throughout the manuscript, let us mention the following useful property of $\mu$:
\begin{align}\label{eq:mu_aux}
  \nabla_v \mu + v \mu = (2\pi )^{-\tfrac{d}{2}}e^{-\frac{\|v\|^2}{2}}\cdot(\tfrac{-2 v}{2})+v\cdot (2\pi )^{-\tfrac{d}{2}}e^{-\frac{\|v\|^2}{2}} = 0.
\end{align}
Similar to \cite{DuaFT10} but with a different form of the perturbation, let us consider
\begin{align}\label{eq:pertb}
 f=\mu + \mu y
\end{align}
with the goal of deriving an equation for $y$ from \eqref{eq:orig_nonl_kfp}. In the following, we address all terms in \eqref{eq:orig_nonl_kfp} individually. We obviously have that
\begin{align}\label{eq:derive_y_aux1}
  \partial_t f = \partial _t(\mu + \mu y ) =\mu \partial_t y.
\end{align}
The second term on the left hand side then is given by
\begin{align}\label{eq:derive_y_aux2}
 v\cdot \nabla_x f = v \cdot \nabla_x(\mu + \mu y) = \mu v \cdot \nabla_x y.
\end{align}
Before we turn to the convolution operators in \eqref{eq:orig_nonl_kfp}, observe that
\begin{equation}\label{eq:convolution_aux}
\begin{aligned}
 \rho_{vf}&=\int_{\mathbb R^d} v f\,\mathrm{d}v = \int_{\mathbb R^d} v(\mu + \mu y)\,\mathrm{d}v = \underbrace{ \int_{\mathbb R^d} v\mu \,\mathrm{d}v }_{=0}+ \int_{\mathbb R^d} \mu vy \,\mathrm{d}v =  \rho_{\mu vy} \\
 \rho_f&= \int_{\mathbb R^d} f \,\mathrm{d}v =\int_{\mathbb R^d} \mu + \mu y \,\mathrm{d}v =\underbrace{\int_{\mathbb R^d} \mu \,\mathrm{d}v}_{=1}+\int_{\mathbb R^d} \mu y\,\mathrm{d}v = 1+\rho_{\mu y}
\end{aligned}
\end{equation}
where in the first line we used that $\mu$ is positive symmetric and $v$ is antisymmetric w.r.t.~the origin. We continue with the third term on the left hand side
\begin{equation}\label{eq:derive_y_aux3}
\begin{aligned}
 U*\rho_{vf} \cdot \nabla_v f &= U* \rho_{\mu v y} \cdot \nabla_v (\mu + \mu y) = U* \rho_{\mu v y} \cdot \left(\nabla_v \mu + y \nabla_v\mu+ \mu \nabla_v y \right) \\
 &= U*\rho_{\mu vy} \cdot (-\mu v-\mu vy+\mu \nabla_vy).
\end{aligned}
\end{equation}
Finally, we use \eqref{eq:mu_aux} to derive
\begin{align*}
  \nabla_v \cdot(\nabla_v f+  vf)&= \nabla_v \cdot( \nabla_v (\mu+\mu y) + v(\mu+\mu y))
  = \nabla_v \cdot(\nabla_v (\mu y) + \mu v y) \\
  &= \nabla_v \cdot(\mu\nabla_v y+ y\nabla_v \mu + \mu vy) = \nabla_v \cdot (\mu \nabla_v y) \\
  &= \mu \Delta_v y+\nabla_v \mu \cdot \nabla _ v y = \mu \Delta_v y -\mu v \cdot \nabla_v y.
\end{align*}
This, together with \eqref{eq:convolution_aux} now yields
\begin{equation}\label{eq:derive_y_aux4}
\begin{aligned}
 U*\rho_f \nabla_v \cdot(\nabla_v f+vf)&= U*(1+\rho_{\mu y}) \mu (\Delta_v y-v\cdot \nabla _v y)\\
 &= (1+U*\rho_{\mu y}) \mu (\Delta_v y-v\cdot \nabla _v y)
\end{aligned}
\end{equation}
where in the last step we used that $U*1=\int_{\mathbb R^d} U(x)\,\mathrm{d}x =1$. Combining \eqref{eq:derive_y_aux1},\eqref{eq:derive_y_aux2},\eqref{eq:derive_y_aux3},\eqref{eq:derive_y_aux4} and eliminating the common factor $\mu$, for $u\equiv 0$, we arrive at
\begin{equation}\label{eq:sharp_nonl_kfp}
\begin{aligned}
  &\partial_t y + v\cdot \nabla_x y + U*\rho_{\mu vy} \cdot (\nabla_v y-yv-v) \\
  &\qquad =\Delta_v y -v \cdot \nabla_v y + U*\rho_{\mu y} (\Delta_vy - v\cdot \nabla_v y).
 \end{aligned}
\end{equation}
In particular, the above system takes the form
\begin{align}\label{eq:RL_nonl_kfp}
 \partial _t y= Ay+Dy- h_1(y) - h_2(y)
\end{align}
where $A,D,h_1$ and $h_2$ are given as
\begin{equation}\label{eq:RL_op}
\begin{aligned}
  Ay &=\Delta_v y - v\cdot \nabla_v y - v \cdot \nabla_xy,  \ \
  Dy=  U*\rho_{\mu yv} \cdot v, \\[1ex]
  h_1(y) &= U*\rho_{\mu y} R_0y, \ \ h_2(y)=U*\rho_{\mu yv} \cdot (\nabla_v y-yv), \\[1ex]
  R_0y &= -\Delta_v y + v\cdot \nabla_v y.
\end{aligned}
\end{equation}
In particular, note that $R_0$ is formally self-adjoint and non negative w.r.t.~the $L^2_\mu$ inner product, see \cite[Remark 2.2]{BreK23}. For later reference, we point out that the formal $L^2_\mu$ adjoint of $A$ is given by
\begin{align}\label{eq:adj_A}
  A^*y = \Delta_v y - v\cdot \nabla_v y + v \cdot \nabla_xy.
\end{align}

Now we turn to the control operator in \eqref{eq:orig_nonl_kfp}. For fixed $u$, the transformation corresponding to \eqref{eq:pertb} is given by
\begin{align*}
 u (\alpha \cdot \nabla_v (\mu +\mu y)) &= u(\alpha\cdot \nabla_v\mu +y  \alpha \cdot \nabla_v \mu+ \mu \alpha \cdot \nabla_v y ) \\
 &= \mu u(-\alpha\cdot v -y \alpha \cdot v+ \alpha \cdot \nabla_v y ) .
\end{align*}
Consequently, the controlled analogue of \eqref{eq:RL_nonl_kfp} is given by
\begin{align}\label{eq:RL_nonl_kfp_con}
 \partial _t y= Ay+Dy- h_1(y) - h_2(y) + uNy + Bu
\end{align}
where $N \in \mathcal{L}(Y,V_v')$, see \cite[Eq.~(3.4)]{BreK23} and $B\in Y$ are defined as
\begin{align}\label{eq:N_B_ops}
 Ny=-y \alpha \cdot v+ \alpha \cdot \nabla_v y,\quad B=-\alpha \cdot v,
\end{align}
and the mappings $h_i$ were defined in \eqref{eq:RL_op}. 
In section \ref{sec:lin_eq} we start by first analyzing the linearization of \eqref{eq:RL_nonl_kfp}.

\paragraph{Notation.} By $C^\infty_0(\mathbb{R}^{2d})$ we denote the set of all functions in $C^\infty(\mathbb{R}^{2d})$ with compact support in $\mathbb{R}^{2d}$. For a linear closed, densely defined operator $A$ with domain $\mathcal{D}(A)$ in a Hilbert space $Z$, we write $A\colon \mathcal{D}(A)\subset Z\to Z$. For $\mathcal{D}(A)$ endowed with the graph norm and the Hilbert space adjoint $A^*$ of $A$ in $Z$, the associated duality pairing $\langle \cdot ,\cdot \rangle _{\mathcal{D}(A^*),[\mathcal{D}(A^*)]'}$ will simply be denoted as $\langle \cdot,\cdot \rangle _{\mathcal{D}}$. Throughout the paper, we will extensively use the weighted (Hilbert) spaces
\begin{equation*}
\begin{aligned}
&Y\!=\!L^2_\mu(\mathbb R^{2d})\! =\! \left\{ y\colon \mathbb R^{2d}\! \to\! \mathbb R \ | \ \mu^{\frac{1}{2}} y \in L^2(\mathbb R^{2d}) \right\},  \| y\| _Y \!= \!\left( \int_{\mathbb R^{2d}} \mu y^2 \, \mathrm{d}x\, \mathrm{d}v\right)^{\frac{1}{2}}, \\[1ex]
&V\!=\!H_\mu^1(\mathbb R^{2d})\!=\!\left\{ y\colon \mathbb R^{2d}\!\to \!\mathbb R\ | \ y\in Y, \nabla y \in Y^{2d} \right\}, \| y\|_V \!= \! \left(\|y\|^2_Y + \|\nabla y \|_{Y^{2d}}^2\right)^{\frac{1}{2}},  \\[1ex]
 & V_v\!=\! H_{\mu,v}^1(\mathbb R^{2d})\!=\!\left\{ y\colon \mathbb R^{2d}\!\to\! \mathbb R\ |  \  y\in Y, \nabla_v y \in Y^d \right\},   \| y\|_{V_v} \!\!=\!\! \left(\|y\|^2_Y \!+ \!\|\nabla_v y \|_{Y^d}^2\right)^{\frac{1}{2}},
 \end{aligned}
 \end{equation*}
where $\nabla_vy=\begin{pmatrix} \tfrac{\partial y}{\partial v_1},\dots, \tfrac{\partial y}{\partial v_d}\end{pmatrix}^\top$.

\section{The linear equation}\label{sec:lin_eq}

In this section we focus on  the abstract linear system
\begin{align}\label{eq:abs_lin_sys}
 \dot{y}(t)&= (A+D)y(t) + g(t), \ \ y(0)=y_0 \in Y
\end{align} 
with $g\in L^2(0,T;V_v')$. Later on $g$ will be replaced by the control terms and the nonlinear terms $h_1,h_2$, cf.~\eqref{eq:RL_nonl_kfp_con}.

From \cite[Theorem 2.1]{ConG10} for vanishing potential ($\Phi\equiv 0$ in the notation used therein), it follows that $A$ generates a contraction semigroup on $Y$. In the following lemma it will be argued that $D\in \mathcal{L}(Y)$. Consequently  $A+D$ as well generates a semigroup $e^{(A+D)t}$, see, e.g., \cite[Theorem 3.2.1]{CurZ95}.

\begin{lemma}\label{lem:D_is_bounded}
  For $Dy=  U*\rho_{\mu yv} \cdot v$ it holds that $D\in \mathcal{L}(Y)$.
\end{lemma}
\begin{proof}
Let us first consider $\rho_{\mu y v}$ for which we have
\begin{align*}
  \| \rho_{\mu y v} \|_{(L^2(\mathbb R^d))^d} &= \left(\int_{\mathbb R^d} \| \rho_{\mu yv}(x)\|_{\mathbb R^d}^2\,\mathrm{d}x  \right)^{\frac{1}{2}}  = \left(\int_{\mathbb R^d} \left\|\int_{\mathbb R^d} v \mu(v) y(x,v)\,\mathrm{d}v \right\|_{\mathbb R^d}^2 \mathrm{d}x \right)^{\frac{1}{2}}.
\end{align*}
An application of the Minkowski integral inequality yields
\begin{align*}
 \| \rho_{ \mu y v}\|_{(L^2(\mathbb R^d))^d} &\le  \int_{\mathbb R^d}\left( \int_{\mathbb R^d}\left\| v \mu(v) y(x,v)\right\|_{\mathbb R^d}^2\,\mathrm{d}x   \right)^{\frac{1}{2}} \mathrm{d}v \\
 &=  \int _{\mathbb R^d} \|v\mu^{\frac{1}{2}}(v)\|_{\mathbb R^d} \left( \int_{\mathbb R^d} | \mu^{\frac{1}{2}}(v) y(x,v)|^2\,\mathrm{d}x \right)^{\frac{1}{2}} \mathrm{d}v.
\end{align*}
The Cauchy-Schwarz inequality for the $v$ variable now leads to
\begin{equation}\label{eq:derive_y_aux5a}
\begin{aligned}
 \| \rho_{v \mu y}\|_{(L^2(\mathbb R^d))^d} &\le  \left(\int_{\mathbb R^d} \| v\mu^{\frac{1}{2}}(v)\|_{\mathbb R^d}^2\,\mathrm{d}v \right)^{\frac{1}{2}} \left(\int_{\mathbb R^d} \int_{\mathbb R^d} | \mu^{\frac{1}{2}}(v)y(x,v)|^2\, \mathrm{d}x\, \mathrm{d}v \right)^{\frac{1}{2}} \\
 &\le \sqrt{d}  \left(\int_{\mathbb R^d} \int_{\mathbb R^d}\mu (v) | y(x,v)|^2\, \mathrm{d}x\, \mathrm{d}v \right)^{\frac{1}{2}} = \sqrt{d} \| y\|_{Y},
\end{aligned}
\end{equation}
where we have used that
\begin{align*}
 \int_{\mathbb R^d} \|v\mu^{\frac{1}{2}}\|^2_{\mathbb R^d} \,\mathrm{d}v = 
 \int_{\mathbb R^d} v \cdot v\mu \,\mathrm{d}v = -
 \int_{\mathbb R^d} v \cdot \nabla_v \mu \,\mathrm{d}v = \int_{\mathbb R^d} \mathrm{div}_v(v) \mu \ \mathrm{d}v =d.
\end{align*}
Consequently, with Young's convolution inequality we obtain
  \begin{align*}
    \|U*\rho_{\mu yv}\|_{(L^2(\mathbb R^d))^d} \le \underbrace{\| U\|_{L^1(\mathbb R^d)} \|}_{=1} \rho_{\mu yv}\|_{(L^2(\mathbb R^d))^d} \le \sqrt{d} \| y\|_{Y}.
  \end{align*}
 Finally, we arrive at 
\begin{align*}
 \|Dy\|_Y^2 &= \int_{\mathbb R^d}\int_{\mathbb R^d} | (U*\rho_{\mu y v})(x) \cdot v|^2 \mu(v)\,\mathrm{d}x \, \mathrm{d}v \\
 &\le  \int_{\mathbb R^d}\int_{\mathbb R^d} \| (U*\rho_{\mu y v})(x)\|^2\cdot \|v\|^2 \mu(v)\,\mathrm{d}x \, \mathrm{d}v \\
 &= \int_{\mathbb R^d} \| U*\rho _{\mu y v}(x) \|^2 \,\mathrm{d}x \int_{\mathbb R^d} \| v\|^2\mu(v)\,\mathrm{d}v \le \| y\|_Y^2 .
\end{align*}
\end{proof}

On $C_0^\infty(\mathbb R^{2d})$, let us consider the operator $R$ defined by
\begin{align*}
 Ry := -\Delta_v y +v \cdot \nabla _v y + y.
\end{align*}
Note that utilizing \eqref{eq:mu_aux} we obtain
\begin{equation}\label{eq:sqrt_R_aux_1}
\begin{aligned}
 \langle Ry,y\rangle_Y&=\int_{\mathbb R^{2d}} -\mu y \Delta_v y + \mu yv \cdot \nabla_v y + \mu y^2 \,\mathrm{d}x\,\mathrm{d}v  \\
 &= \int_{\mathbb R^{2d}} \nabla_v(\mu y) \cdot \nabla_v y + \mu yv \cdot \nabla_vy + \mu y^2\,\mathrm{d}x\,\mathrm{d}v \\
 &= \int_{\mathbb R^{2d}} \mu \nabla_v y \cdot \nabla_v y+ \underbrace{y\nabla_v \mu \cdot \nabla_v y + \mu yv \cdot \nabla_vy}_{=0} + \mu y^2\,\mathrm{d}x\,\mathrm{d}v \\
 &= \|\nabla_v y\|_{Y^d}^2 + \|y\|_Y^2 = \|y\|_{V_v}^2 = \langle y,Ry\rangle_Y
\end{aligned}
\end{equation}
showing that $R$ is symmetric and coercive on $C_0^\infty(\mathbb R^{2d})$. It is therefore closable as an operator in $Y$, see \cite[Chapter II, Proposition 3.14 ]{EngN99}, and we may consider $\overline{R}$. In fact, this operator $\overline{R}$ is self-adjoint in $Y$. This can be verified utilizing \cite[Problem V.3.17]{Kat80}. For this it suffices to show that $\mathrm{range}(R-i\zeta )$ is dense in $Y$, for some positive and some negative $\zeta$. For the case of the classical Fokker-Planck operator with $y$ only depending on one variable and a smooth potential $V(v)$, given by $V(v)=\frac{1}{2}\|v\|^2$ in our case, we can refer to \cite[Corollary 3.2.2, Appendix A.5]{BakGL14}. The desired density in the  case  with $y$ a function of $x$ and of $v$ can then  be obtained by observing that for 
$\{\varphi_j\}_{j=1}^{\infty}$ dense in $L^2_\mu(\mathbb{R^d})$ and   $\{\psi_k\}_{k=1}^{\infty}$  dense in  $H^1_\mu(\mathbb{R^d})$, 
the set $ \{\sum_{j=1,k=1}^{n,m} \varphi_j(x) \psi_k(v): n\in \mathbb{N}, m\in \mathbb{N} \} $ is dense in $V_v$ .

From, e.g., \cite[Examples A.4.2/A.4.3]{CurZ95} and \cite[Problem V.3.32, Theorem V.3.35]{Kat80} it is well-known that $\overline{R}^{-1}\in\mathcal{L}(Y)$ is also a (uniformly) positive self-adjoint operator for which there exists a (uniformly) positive self-adjoint square root $\overline{R}^{-\frac{1}{2}}:=(\overline{R}^{-1})^{\frac{1}{2}}$. In fact, we also have that
$\overline{R}^{-\frac{1}{2}}=(\overline{R}^{\frac{1}{2}})^{-1}$, where $\overline{R}^{\frac{1}{2}}$ is the unique (uniformly) positive self-adjoint square root of $\overline{R}$, i.e., $\overline{R}y = \overline{R}^{\frac{1}{2}}\overline{R}^{\frac{1}{2}}y$ and $\mathcal{D}(\overline{R}^{\frac{1}{2}})\supset \mathcal{D}(\overline{R})$.

Since $\overline{R}^{-1}\in\mathcal{L}(Y)$, we know that $\overline{R}^{-\frac{1}{2}}\in\mathcal{L}(Y)$. We even have $\overline{R}^{-\frac{1}{2}}\in\mathcal{L}(Y,V_v)$. Indeed, for $y \in Y$, consider $z=\overline{R}^{-\frac{1}{2}}y$ and note that
\begin{align*}
 \|y\|_Y^2=\langle y,y\rangle_Y = \langle \overline{R}^{\frac{1}{2}}z,\overline{R}^{\frac{1}{2}}z\rangle_Y =\langle \overline{R}z,z\rangle_Y \stackrel{\text{\eqref{eq:sqrt_R_aux_1}}}{=} \|z\|_{V_v}^2 = \| \overline{R}^{-\frac{1}{2}}y\|_{V_v}^2.
\end{align*}
We then also obtain $(\overline{R}^{-\frac{1}{2}})'\in\mathcal{L}(V_v',Y)$ which, since $
 (\overline{R}^{-\frac{1}{2}})'=((\overline{R}^{\frac{1}{2}})^{-1})'=((\overline{R}^{\frac{1}{2}})')^{-1}$ and $\overline{R}^{\frac{1}{2}}$ is self-adjoint finally implies $\overline{R}^{-\frac{1}{2}}\in\mathcal{L}(V_v',Y)$.

Let us return to \eqref{eq:abs_lin_sys} which after the above considerations we interpret as
\begin{align*}
 \dot{y}(t)&= (A+D)y(t) +\overline{R}^{\frac{1}{2}} w(t)
\end{align*}
where $w(t)=\overline{R}^{-\frac{1}{2}}g(t)$. In particular, we emphasize that $w\in L^2(0,T;Y)$ and $\overline{R}^{\frac{1}{2}}\in \mathcal{L}(Y,V_v')$. The following result asserts that while $\overline{R}^{\frac{1}{2}}$ is not bounded w.r.t.~the state space $Y$, it is nevertheless compatible with the semigroup generated by $A+D$ in the sense of \emph{admissibility}. For an introduction to this control theoretic concept, we refer to, e.g., \cite{TucW09}.

\begin{lemma}\label{lem:sqrtR_is_admissible}
 $\overline{R}^{\frac{1}{2}}$ is an admissible control operator for $e^{(A+D)t}$.
\end{lemma}
\begin{proof}
Since $D\in \mathcal{L}(Y)$, with \cite[Theorem 5.4.2]{TucW09} it follows that $\overline{R}^{\frac{1}{2}}$ is admissible for $e^{(A+D)t}$ if and only if $\overline{R}^{\frac{1}{2}}$ is admissible for $e^{(A-\alpha I)t}$ for any $\alpha \in \mathbb R$. Since infinite-time admissibility implies admissibility, we may use \cite[Theorem 5.1.1 (d)]{TucW09} to prove the assertion. For this purpose, observe that for arbitrary $z \in C_0^\infty(\mathbb R^{2d})$ it holds that
\begin{align*}
 \langle  z, A^*z\rangle_Y &=  \langle Az,z\rangle_Y = \int_{\mathbb R^{2d}} (z\Delta_v z - z v \cdot \nabla_v z - zv \cdot \nabla_x z) \mu\,\mathrm{d}x\,\mathrm{d}v  \\
 &= \int_{\mathbb R^{2d}} -\mu \| \nabla_v z \|_{\mathbb R^d}^2-z \nabla_v z \cdot \underbrace{\nabla_v \mu}_{=-v\mu} - \mu z v \cdot \nabla_v z - \mu z v \cdot \nabla_x z \,\mathrm{d}x\,\mathrm{d}v \\
 &= -\int _{\mathbb R^d} \mu \| \nabla_v z \|_{\mathbb R^d}^2 + \frac{1}{2}\mu zv \cdot \nabla_x z+\frac{1}{2}\mu zv \cdot \nabla_x z \,\mathrm{d}x\, \mathrm{d}v \\
 &=-\int _{\mathbb R^d} \mu \| \nabla_v z \|_{\mathbb R^d}^2 + \frac{1}{2}\mu zv \cdot \nabla_x z-\frac{1}{2}\mu (\nabla_x z \cdot v) z \,\mathrm{d}x\, \mathrm{d}v \\
 &= -\int _{\mathbb R^d} \mu \| \nabla_v z \|_{\mathbb R^d}^2\,\mathrm{d}x\, \mathrm{d}v,
\end{align*}
where in the second to last step we  used that $\mu$ and $v$ are  independent of $x$. For  $\overline{R}^{\frac{1}{2}}\in \mathcal{L}( Y,V_v')$, consider $(\overline{R}^{\frac{1}{2}})'\in \mathcal{L}(V_v,Y)$. From \cite[Lemma 2.3, Lemma 3.2]{BreK23} (applied with $G\equiv 0$) we have that
$\mathcal{D}(A^*)\subset V_v$ and that $C_0^\infty(\mathbb R^{2d})$ is dense in $\mathcal{D}(A^*)$. With the above computations we thus obtain for $\alpha=\|(\overline{R}^{\frac{1}{2}})'\|_{\mathcal{L}(V_v,Y)}^2$ and for all $z \in \mathcal{D}(A^*)$ that
\begin{align*}
  \langle  \alpha z,(A-I )z \rangle_Y&= - \alpha (\| z\|_Y^2 + \| \nabla_v z \|_{Y^d}^2)  = -\alpha \|z\|_{V_v}^2 = -\|(\overline{R}^{\frac{1}{2}})'\|_{\mathcal{L}(V_v,Y)}^2 \|z\|_{V_v}^2 \le -  \| (\overline{R}^{\frac{1}{2}})' z\|_Y^2  .
\end{align*}
Since $\alpha I\succ 0$, \cite[Theorem 5.1.1 (d)]{TucW09} shows that $(\overline{R}^{\frac{1}{2}})'$ is an infinite-time admissible observation operator for $e^{(A-\alpha I)^*t}$. Consequently, $(\overline{R}^{\frac{1}{2}})'$ is an admissible observation operator for $e^{(A+D)^*t}$ and with \cite[Theorem 4.4.3]{TucW09} this is equivalent to $\overline{R}^{\frac{1}{2}}$ being an admissible control operator for $e^{(A+D)t}$.
\end{proof}

Following \cite[Definition 4.1.1]{TucW09}, for $g \in L^2(0,T;Y_{-1})$, we refer to $y$ as a solution of \eqref{eq:abs_lin_sys} in $Y_{-1}:=\mathcal{D}(A^*)'$ if $y \in L^1(0,T;Y)\cap C([0,T];Y_{-1})$ and it satisfies for every $t\ge 0$ 
\begin{align}\label{eq:strong_sol_int}
  y(t)-y_0=\int_0^t (A+D)y(s) + g(s) \, \mathrm{d}s \ \ \text{in } Y_{-1},
\end{align}
which is equivalent to 
\begin{align}\label{eq_strong_sol_var}
\langle y(t)-y_0,\psi \rangle_{\mathcal{D}} = \int_0^t \left[ \langle y(s),(A+D)^* \psi \rangle _Y + \langle g(s),\psi \rangle_{\mathcal{D}} \right] \, \mathrm{d}s
\end{align}
for every $t\ge 0$ and every $\psi \in \mathcal{D}(A^*)$, see \cite[Remark 4.1.2]{TucW09}. If $y$ is a solution in $Y_{-1}$, then $y$ is given by 
\begin{align}\label{eq:mild_sol}
 y(t) = e^{(A+D)t}y_0 + \int_0^t e^{(A+D)(t-s)}g(s)\,\mathrm{d}s,
\end{align}
see \cite[Proposition 4.1.4]{TucW09}.

\begin{corollary}\label{cor:mild_sol_lin_sys}
  For every $y_0\in Y, g\in L^2(0,T;V_v')$ the initial value problem
  \begin{align}\label{eq:lin_per}
   \dot{y}(t) = (A+D)y(t)+g(t), \ \ y(0)=y_0
  \end{align}
has a unique solution $y \in C([0,T];Y) \cap H^1(0,T;Y_{-1})$. 
Moreover, there exists a constant $C$, only depending on $T$ such that
\begin{align*}
 \|y\|_{L^\infty(0,T;Y)} \le C ( \| y_0\|_Y + \|g\|_{L^2(0,T;V_v')} ).
\end{align*}
\end{corollary}
\begin{proof}
First, we observe that $V_v'\subseteq \mathcal{D}(A^*)'$. This follows from \cite[Lemma 2.3]{BreK23} by setting $G\equiv 0$ and $\overline{\mathcal{L}}=A^*$ therein.
Since $\overline{R}^{\frac{1}{2}} \in \mathcal{L}(Y,V_v')$ is admissible, we can consider the controllability map 
 \begin{align*}
  \Phi^t\colon w \mapsto \Phi^t w &= \int_0^t e^{(A+D)(t-s)}\overline{R}^{\frac{1}{2}}w(s)\,\mathrm{d}s
 \end{align*}
which satisfies 
$\Phi^t \in \mathcal{L}(L^2(0,t;Y),Y).$
Setting $w(t)=\overline{R}^{-\frac{1}{2}}g(t)$, this implies that \eqref{eq:lin_per} admits a unique solution with the specified regularity properties, see \cite[Proposition 4.2.5]{TucW09}.

 Since a solution $y \in Y_{-1}$ is also a mild solution, we obtain
\begin{align*}
 y(t) &= e^{(A+D)t}y_0 + \int_0^t e^{(A+D)(t-s)}g(s)\,\mathrm{d}s =e^{(A+D)t}y_0 +  \Phi^t (\overline{R}^{-\frac{1}{2}}g)
\end{align*}
with $\overline{R}^{-\frac{1}{2}}\in \mathcal{L}(V_v',Y)$ as before.
It is known that controllability maps are non decreasing \cite[Proof of  Proposition 4.2.4]{TucW09} in the sense that
\begin{align}\label{eq:conmap_nondec}
 \|\Phi^{t_1}\|_{ \mathcal{L}(L^2(0,t_1),Y)}\le
 \|\Phi^{t_2}\|_{ \mathcal{L}(L^2(0,t_2),Y)}, 
\end{align}
for all $t_1\le t_2.$ We therefore arrive at
\begin{align*}
 \|y(t)\|_Y &\le \|e^{(A+D)t}y_0\|_Y +  \|\Phi^t ( \overline{R}^{-\frac{1}{2}}g) \|_Y \\
 &\le \|e^{(A+D)t}y_0\|_Y + \|\Phi^t  \|_{\mathcal{L}(L^2(0,t;Y),Y)} \|\overline{R}^{-\frac{1}{2}}g\|_{L^2(0,t;Y)} \\
 &\le c_1\|y_0\|_Y + c_2 \|\overline{R}^{-\frac{1}{2}}\|_{\mathcal{L}(V_v',Y)} \| g\|_{L^2(0,T;V_v')}
\end{align*}
where in the last step we used \eqref{eq:conmap_nondec}.
\end{proof}

For the remainder of the manuscript, we define
\begin{align*}
 \VERT y \VERT := \| y\|_{L^2(0,T;V_v)}+\|y\|_{L^{\infty}(0,T;Y)}.
\end{align*}

\begin{proposition}\label{prop:more_reg} Let $y$ denote the solution to \eqref{eq:lin_per} from Corollary \ref{cor:mild_sol_lin_sys}. Then, for all $T>0$ there exists a constant $\widehat{C}=\widehat{C}(T)$ s.t.
  \begin{align}
  \mathrm{max}( \VERT y\VERT,\|\dot y\|_{L^2(0,T;Y_{-1})} )  &\le \widehat{C} \left( \| y_0\|_Y + \| g\|_{L^2(0,T;V_v')} \right).
\end{align}
\end{proposition}

\begin{proof}
  For the proof, we can follow the arguments provided in the proof of \cite[Proposition 3.3]{BreK23}. For this purpose, we define $A_{\varepsilon} \in \mathcal{L}(V,V')$ as
  \begin{align*}
   A_{\varepsilon} := A + \varepsilon \Delta_x
  \end{align*}
  where $\varepsilon>0$ and recall from \cite{BreK23} that it generates an analytic semigroup in $Y$ which in particular satisfies
  \begin{align}\label{eq:dual_pair}
   \langle A_\varepsilon y,y\rangle_{V',V} = -\| \nabla _v y\|_Y^2 - \varepsilon \| \nabla _x y\|_Y^2
  \end{align}
  for all $ y\in V$. We next consider the perturbed equation
  \begin{align}\label{eq:lin_per_eps}
   \dot{y}_{\varepsilon}(t) = (A_\varepsilon+D)y_{\varepsilon}(t)+g(t), \ \ y_{\varepsilon}(0)=y_0.
  \end{align}
  Taking the inner product with $\mu y_{\varepsilon},$ integrating over $\mathbb R^{2d}$ and following similar calculations as in \cite{BreK23}, utilizing \eqref{eq:dual_pair}, we arrive at
  \begin{align*}
   \tfrac{1}{2} \tfrac{\mathrm{d}}{\mathrm{d}t} \| y_{\varepsilon} \|_Y^2 &\le - \tfrac{1}{2}\|\nabla_v y_{\varepsilon}\|_Y^2 - \varepsilon\|\nabla_x y_{\varepsilon}\|^2 + \|D\|_{\mathcal{L}(Y)} \| y_{\varepsilon}\|_Y^2+\tfrac{1}{2}\| g\|_{V_v'}^2+\tfrac{1}{2}\| y_{\varepsilon}\|_{Y}^2.
  \end{align*}
  This implies that
  \begin{align*}
&   \| y_\varepsilon(t)\|_Y^2 + \| \nabla_v y_{\varepsilon}\|^2_{L^2(0,t;Y)}+ 2 \varepsilon \| \nabla_v y_{\varepsilon}\|^2_{L^2(0,t;Y)} \\
&\quad \le \| y_{\varepsilon}(0)\|_Y^2+(2 \| D\|_{\mathcal{L}(Y)}+1) \| y_{\varepsilon}\|_{L^2(0,t;Y)}^2 + \| g\|_{L^2(0,t;V_v')}^2.
  \end{align*}
  By Gronwall's inequality this implies that
  \begin{align*}
   \max(\|y_{\varepsilon}\|_{L^\infty(0,T;Y)}, \| \nabla_v y_{\varepsilon} \|_{L^2(0,T;Y)})\le C (\| y_0\|_Y + \| g\|_{L^2(0,T;V_v')}).
  \end{align*}
  Now we can pass to the limit $\varepsilon \to 0$ as in \cite{BreK23} and we arrive at the assertion.
\end{proof}

\section{The nonlinear equation}\label{sec:nonlin_eq}

In this section existence of a solution to the nonlinear equation \eqref{eq:sharp_nonl_kfp} and the controlled equation \eqref{eq:RL_nonl_kfp_con}  completed with the initial condition $y(0)=y_0$ are proven. 
We first  provide estimates for $h_1$ and $h_2$. 

\begin{lemma}\label{lem:loc_lipschitz_stat}
  Let $y,z \in V_v$. Then it holds that
 \begin{align*}
   \| h_1(y)-h_1(z) \|_{V_v'} &\le \|U\|_{L^2(\mathbb R^d)}(\| y\|_Y  \|y-z\|_{V_v} + \|y-z\|_Y \| z\|_{V_v}) , \\
   \| h_2(y)-h_2(z) \|_{V_v'} &\le \sqrt{d} \|U\|_{L^2(\mathbb R^d)}(\| y\|_Y \|y-z\|_Y +  \|y-z\|_{Y}\|z\|_Y).
  \end{align*}
  \end{lemma}
  The proof is given in the appendix. The following lemma follows easily from the previous one.

\begin{lemma}\label{lem:loc_lipschitz}
  For $y,z \in L^2(0,T;V_v) \cap L^{\infty}(0,T;Y)$ the following estimates hold: 
   \begin{align*}
   \| h_1(y)-h_1(z) \|_{L^2(0,T;V_v')} &\le \|U\|_{L^2(\mathbb R^d)}\left(\|y\|_{L^\infty(0,T;Y)} \|y-z\|_{L^2(0,T;V_v)}\right) \\
   &\qquad+\|y-z\|_{L^\infty(0,T;Y)}\|z\|_{L^2(0,T;V_v)})
   \\
   &\le \|U\|_{L^2(\mathbb R^d)}(\VERT y \VERT + \VERT z\VERT ) \VERT y-z\VERT, \\
   \| h_2(y)-h_2(z) \|_{L^2(0,T;V_v')} &\le \sqrt{d}\|U\|_{L^2(\mathbb R^d)}\left(\|y\|_{L^\infty(0,T;Y)}+\|z\|_{L^\infty(0,T;Y)}\right) \|y-z\|_{L^2(0,T;Y)}\\
   &\le \sqrt{d}\|U\|_{L^2(\mathbb R^d)}(\VERT y \VERT + \VERT z\VERT ) \VERT y-z\VERT.
  \end{align*}
\end{lemma}

We are now prepared to assert the existence of a solution to the nonlinear equation under a smallness assumption 
on the initial data. 

\begin{proposition}\label{prop:nonlinear}
 For $y_0\in Y$ such that $\|y_0\|_Y \le \mu:=
 \frac{3}{32\sqrt{d}\|U\| \widehat{C}^2} $
 there exists a solution $y\in C([0,T];Y)\cap L^2(0,T;V_v) \cap W^{1,2}(0,T;Y_{-1})$ to
  \begin{align}\label{eq:nonlin_per}
   \dot{y}(t) = (A+D)y(t)- h_1(y(t))-h_2(y(t)), \ \ y(0)=y_0
  \end{align}
  which for every $t\in [0,T]$ satisfies 
\begin{align}\label{eq:nonl_strong_int}
 y(t)-y_0 = \int_0^t ((A+D)y(s) - (h_1(y(s))+h_2(y(s))))\,\mathrm{d}s \ \ \text{ in } Y_{-1}.
\end{align}
\end{proposition}

\begin{proof}
  The statement will be shown utilizing a classical fixed point argument. For this purpose, let us define the set
  \begin{align*}
  \mathcal{F}:=\left\{ y\in L^2(0,T;V_v)\cap C([0,T];Y)  \left| \VERT y \VERT \le \kappa:= 
  \frac{1}{8\sqrt{d}\|U\|\widehat{C}}     \right. \right\}.
\end{align*}
We further define the mapping $\mathcal{T}\colon \mathcal{F} \to L^2(0,T;V_v)\cap C([0,T];Y),  z \mapsto
y_z,$ where $y_z$ is the solution to
\begin{align*}
  \dot{y}_z(t) &= (A+D)y_z(t)- h_1(z(t))-h_2(z(t)), \ \ y_z(0)=y_0.
\end{align*}
It is well-defined by Proposition \ref{prop:more_reg} with $g=-(h_1(z)+h_2(z))$, using that Lemma \ref{lem:loc_lipschitz} implies $h_i(z)\in L^2(0,T;V_v'), i=1,2$. In addition, by the choice of $\mu$ and $\kappa$ we have the estimate
\begin{align*}
  \VERT y_z \VERT \le \widehat{C} ( \| y_0\|_Y + 2\sqrt{d}\|U\|\VERT z\VERT^2  )\le \widehat{C} (\mu +2\sqrt{d}\|U\|\kappa^2) \le \kappa .
\end{align*}
Next, for $z_1,z_2 \in \mathcal{F}$, the difference $e:=y_{z_1}-y_{z_2}$ satisfies
\begin{align*}
\dot{e}(t) = (A+D)e(t)- (h_1(z_1(t))-h_1(z_2(t)))
- (h_2(z_1(t))-h_2(z_2(t))), \   e(0)=0
\end{align*}
which, again by Proposition \ref{prop:more_reg} and Lemma \ref{lem:loc_lipschitz} yields
\begin{align*}
   \VERT y_{z_1}-y_{z_2}\VERT \le 4\sqrt{d} {\widehat{C}}\|U\|\kappa \VERT z_1-z_2{\VERT \le \tfrac{1}{2} \VERT z_1-z_2 \VERT}.
\end{align*}
Consequently, $\mathcal{T}$ is a contraction on $\mathcal{F}$ and implies the existence of a unique solution to \eqref{eq:nonlin_per} in the set $\mathcal{F}$.

Since $y\in L^2(0,T;V_v)$, we know that $h_1(y)+h_2(y)\in L^2(0,T;V_v')$ such that Corollary \ref{cor:mild_sol_lin_sys} yields $y\in W^{1,2}(0,T;Y_{-1})$.
\end{proof}

We now turn to the controlled nonlinear equation \eqref{eq:RL_nonl_kfp_con}. It will be convenient to introduce the space
\begin{align}\label{def:X-space}
 \mathcal{X}:= C([0,T];Y)\cap L^2(0,T;V_v) \cap W^{1,2}(0,T;Y_{-1}).
\end{align}

\begin{definition}\label{def:nonl_strong_int_cont}
We call $y$ a solution to \eqref{eq:RL_nonl_kfp_con} in $Y_{-1}$ if $y\in  \mathcal{X}$ and it satisfies
\begin{align}\label{eq:nonl_strong_int_cont}
 y(t)-y_0 = \int_0^t (A+D)y(s)+Ny(s)u(s)+Bu(s) - (h_1(y(s))+h_2(y(s)))\,\mathrm{d}s
\end{align}
 in  $Y_{-1}$.
\end{definition}
 
The following result will later be used for the optimal control problem and builds upon the existence of solutions to \eqref{eq:RL_nonl_kfp_con} satisfying an $L^2(0,T;V_v)$ bound.

\begin{proposition}\label{prop:8}
  For $y_0\in Y,$ with $\|y_0\| \le k_0$,  $u\in L^2(0,T)$ with $\|u\|_{L^\infty(0,T)}\le k_1$, and $y \in \mathcal{X}$ a solution to \eqref{eq:nonl_strong_int_cont} with $\|y\|_{ L^2(0,T;V_v)}\le k_2$, for constants $k_0,k_1,k_2 \in$ there exist $C(k_0,k_1,k_2),\tilde{C}(k_0,k_1,k_2)$ such that $\|y\|_{L^\infty(0,T;Y)}\le C(k_0,k_1,k_2)$ and $\|y\|_{W^{1,2}(0,T;Y_{-1})}\le \tilde{C}(k_0,k_1,k_2)$ with the property that for each $k_2>0$ fixed, $C(k_0,k_1,k_2)\to 0$ for $k_0+k_1\to 0$. 
\end{proposition}
\begin{proof}
Similar to the proof of Corollary \ref{cor:mild_sol_lin_sys}, we consider the following controllability maps
  \begin{align*}
  \Phi_1^t\colon w \mapsto \Phi_2^t w &= \int_0^t e^{(A+D)(t-s)}\overline{R}^{\frac{1}{2}}w(s)\,\mathrm{d}s \\
  \Phi_2^t\colon u \mapsto \Phi_1^t u &= \int_0^t e^{(A+D)(t-s)}Bu(s)\,\mathrm{d}s
   \end{align*}
   and recall that since $\overline{R}^{\frac{1}{2}}\in \mathcal{L}(Y,V_v')$ and $B\in \mathcal{L}(\mathbb R,Y)$ are admissible we have $\Phi_1^t \in \mathcal{L}(L^2(0,t;Y),Y)$ and $\Phi_2^t \in \mathcal{L}(L^2(0,t),Y).$ With reference to \eqref{eq:nonl_strong_int_cont}, we set $w=Nyu+h_1(y)+h_2(y)$ and estimate for $t\in [0,T]$
   \begin{align*}
    \|y(t)\|_Y&\le \| e^{(A+D)t}y_0\|_Y+\|\Phi_1^t \overline{R}^{-\frac{1}{2}} Nyu\|_Y+\|\Phi_1^t \overline{R}^{-\frac{1}{2}} h_1(y)\|_Y \\
    &\quad +\|\Phi_1^t \overline{R}^{-\frac{1}{2}} h_2(y)\|_Y+ \|\Phi_2^t u\|_Y \\
    &\le Me^{\omega t} \| y_0\|_Y + c_1 \| \overline{R}^{-\frac{1}{2}}N\| _{\mathcal{L}(Y)} \| yu\|_{L^2(0,t;Y)} +c_2 \| u \|_{L^2(0,t)} \\
    &\quad +c_1 \| \overline{R}^{-\frac{1}{2}}\| _{\mathcal{L}(V_v',Y)} (\| h_1(y)\|_{L^2(0,t;V_v')} + \| h_2(y)\| _{L^2(0,t;V_v')}  )
   \end{align*}
   for constants $c_1$ and $c_2$ independent of $y$ and $u$.

   Applying Lemma \ref{lem:loc_lipschitz_stat}, we obtain
   \begin{align*}
   \|y(t)\|_Y&\le Me^{\omega t} \| y_0\|_Y + c_1 \| \overline{R}^{-\frac{1}{2}}N\| _{\mathcal{L}(Y)} \| y\|_{L^2(0,t;Y)} \|u\|_{L^\infty(0,t)} +c_2 \| u \|_{L^2(0,t)} \\
    &\quad +2c_1\sqrt{d}\|U\| \| \overline{R}^{-\frac{1}{2}}\| _{\mathcal{L}(V_v',Y)} \left(\int_0^t \| y(s)\|_Y^2 \| y(s)\|_{V_v}^2 \,\mathrm{d}s\right)^{\frac{1}{2}}\\
    &\le    C(k_0, k_1,k_2) + C\left(\int_0^t \| y(s)\|_Y^2 \| y(s)\|_{V_v}^2 \,\mathrm{d}s\right)^{\frac{1}{2}}.
   \end{align*}
   Squaring both sides yields
   \begin{align*}
    \|y(t)\|_Y^2&\le 2C(k_0,k_1,k_2)^2+ 2C^2 \int_0^t \| y(s)\|_Y^2 \| y(s)\|_{V_v}^2 \,\mathrm{d}s
   \end{align*}
   and by Gronwall's inequality we obtain
   \begin{align*}
    \|y(t)\|_Y^2& \le 2C(k_0, k_1,k_2)^2 e^{2C^2\| y\|_{L^2(0,T;V_v)}^2}
   \end{align*}
   for all $t\in[0,T]$. 
   The estimate on the $W^{1,2}(0,T;Y_{-1})$ norm of the solution can now be obtained by Proposition \ref{prop:more_reg} and Lemma  \ref{lem:loc_lipschitz_stat}.
\end{proof}

In the following theorem  a sufficient condition on $u$ and $y_0$ is given under which existence of a solution to  \eqref{eq:nonl_strong_int_cont} is guaranteed.

  \begin{theorem}\label{thm:weak_non_con}
 For all $y_0\in Y$ and $u\in L^2(0,T)$ sufficiently small in the respective norms, there exists a solution $y\in \mathcal{X}$ to the nonlinear controlled equation \eqref{eq:nonl_strong_int_cont}.
\end{theorem}
\begin{proof}
We follow the proof of Proposition \ref{prop:nonlinear}. Let us thus define 
\begin{align*}
 \kappa := \min((\widehat{C}(4\sqrt{d}\|U\|+1))^{-1},(16\sqrt{d} \|U\|\widehat{C})^{-1})
\end{align*}
and
  \begin{align*}
  \mathcal{F}:=\left\{ y\in L^2(0,T;V_v)\cap C([0,T];Y)  \left| \VERT y \VERT \le \kappa 
  \right. \right\} 
\end{align*}
endowed with the norm of the space $L^2(0,T;V_v)\cap C([0,T];Y)$. We further define the mapping $\mathcal{T}\colon \mathcal{F} \to L^2(0,T;V_v)\cap C([0,T];Y),  z \mapsto
y_z,$ where $y_z$ is the solution to
\begin{align*}
  \dot{y}_z(t) &= (A+D)y_z(t)+Nz(t)u(t)+Bu(t)- h_1(z(t))-h_2(z(t)), \ \ y_z(0)=y_0.
\end{align*}
We now apply Proposition \ref{prop:more_reg} and Young's inequality to estimate
\begin{align*}
 \VERT y_z \VERT \le \widehat{C} (\mu +(2\sqrt{d}\|U\| +\frac{1}{2})\kappa^2) 
\end{align*}
with 
\begin{align}
  \mu= \|y_0\|_Y + \|Bu\|_{L^2(0,T;Y)} + \frac{1}{2} \|N\|^2_{\mathcal{L}(Y,V')} \|u\|_{L^2(0,T)}^2\le \frac{\kappa}{2\widehat{C}}.
\end{align}
By definition of $\kappa$ we observe that
\begin{align*}
 \widehat{C} (\mu +(2\sqrt{d}\|U\| +\frac{1}{2})\kappa^2) 
 &\le \widehat{C} (\mu + (2\sqrt{d}\|U\|+\frac{1}{2})(4\sqrt{d}\|U\|+1)^{-1}\widehat{C}^{-1} \kappa) =\widehat{C}\mu+\frac{1}{2}\kappa \le \kappa
\end{align*}
and consequently $\mathcal{T}$ is mapping $\mathcal{F}$ to itself. For $z_1,z_2 \in \mathcal{F}$ and the difference $e=y_{z_1}-y_{z_2}$ we then obtain
\begin{align*}
\dot{e}(t) &= (A+D)e(t)- (h_1(z_1(t))-h_1(z_2(t)))
- (h_2(z_1(t))-h_2(z_2(t)))+Ne(t)u(t), \\ e(0)&=0
\end{align*}
for which we estimate
\begin{align*}
 \VERT e \VERT &\le \widehat{C} (\sqrt{d}\|U\|4 \kappa +\|u\|_{L^2(0,T)} \|N\|_{\mathcal{L}(Y,V_v')} ) \VERT z_1-z_2\VERT \\
 &\le \frac{1}{4}\VERT z_1-z_2\VERT + \widehat{C} \|u\|_{L^2(0,T)} \|N\|_{\mathcal{L}(Y,V_v')} \VERT z_1-z_2 \VERT \le c \VERT z_1-z_2 \VERT 
\end{align*}
with $c<1$ provided that $\widehat{C} \|u\|_{L^2(0,T)} \|N\|_{\mathcal{L}(Y,V_v')}< \frac{3}{4}$. Hence, $\mathcal{T}$ is a contraction in $\mathcal{F}$. 

Summarizing, Banach's fixed point theorem is applicable if
\begin{align*}
 &\|y_0\|_Y + \|Bu\|_{L^2(0,T;Y)} + \frac{1}{2} \|N\|^2_{\mathcal{L}(Y,V')} \|u\|_{L^2(0,T)}^2 \\ &\quad \le \frac{
  \min((\widehat{C}(4\sqrt{d}\|U\|+1))^{-1},(16\sqrt{d} \|U\|\widehat{C})^{-1})
 }{2\widehat{C}} 
 \end{align*}
 and $
 \widehat{C} \|u\|_{L^2(0,T)} \|N\|_{\mathcal{L}(Y,V_v')}< \frac{3}{4}.
$
\end{proof}

\section{An optimal control problem}\label{sec:opt_con_prob}

The objective of this section is to show that the solution concept for the state equation is appropriate to  consider certain optimal control problems. Concretely we shall investigate the problem
\begin{align}\label{eq:opt_con_prob}
  &\inf_{u\in \mathcal{U}_{\mathrm{ad}}} \mathcal{J}(u):=\frac{1}{2} \int_0^ T \| y(u;t)- y_{\dd}(t)\|_{V_v}^2 \, \mathrm{d}t+ \frac{\beta}{2}\int_0^T u(t)^2 \, \mathrm{d}t, 
  \end{align}
 where $y(u;\cdot)$ is a solution to
 \begin{align}\label{eq:state_eq_opt_con_prob}
      \dot{y} = Ay+Dy- h_1(y) - h_2(y) + uNy + Bu, \text{ in } Y_{-1},  \text{ with } \ y(0)=y_0,
  \end{align}

\begin{align*}
 \mathcal{U}_{\mathrm{ad}} := \{ u \in L^\infty(0,T) \ | \ u^{\mathrm{min}}(t)\le u(t)\le u^{\mathrm{max}}(t) \ \text{ for a.e. } t \in [0,T] \},
\end{align*}
and $y_{\dd} \in L^2(0,T;V_v)$,  $u^{\min}$ and $u^{\mathrm{max}} \in L^\infty(0,T)$ with $u^{\min}(t)\le 0 \le u^{\max}(t)$. We also recall the definitions of the operators $N$ and $B$ from \eqref{eq:N_B_ops} characterizing the control action of the system. Further, throughout this section we assume that

\begin{assumption} \label{ass:U}
$U \in W^{s+1,1}(\mathbb R^d) \cap  H^s(\mathbb R^d)$ for $s>\tfrac{d}{2}$.
\end{assumption}

We have the following existence theorem. 

\begin{theorem}\label{thm:ex_opt_cont}
  If Assumption \ref{ass:U} holds and $y_0$ is sufficiently small, then there exists an optimal control $\bar{u}$ to \eqref{eq:opt_con_prob}.
\end{theorem}

The proofs for this theorem and the following result are given in the appendix.

While the previous theorem asserts existence of an optimal control for \eqref{eq:opt_con_prob}  there is no guarantee for its uniqueness due to the nonlinearity of  the state equation. Moreover the state $y$ associated to an optimal may not be unique globally. The following result, however, guarantees that the state associated to an optimal control-state pair is unique. 

\begin{proposition}\label{prop:unique}
Let Assumption \ref{ass:U} hold and let $\bar{u}$ be an optimal control of \eqref{eq:opt_con_prob}. If $u^{\min}, u^{\max}$ and $y_0$ are sufficiently small, in the sense that
\begin{align*}
\|y_0\|&\le \mu \\
u^\infty:=\max(\|u^{\min}\|_{L^\infty(0,T)},\|u^\infty\|_{L^\infty(0,T)})&\le 1 \\
 8d\widehat{C}^2\|U\|_{L^2(\mathbb R^d)}^2 (1+\|N\| )(2\|y_{\mathrm{d}}\|+\kappa) C(\|y_0\|,u^\infty,\kappa+2\|y_{\mathrm{d}}\|)&\le \frac{1}{2} \\
 4\sqrt{d}\widehat{C}^2\|U\|_{L^2(\mathbb R^d)} (1+\|N\| )\big(\|y_0\|+(2\|y_{\mathrm{d}}\|+\kappa)u^\infty\|N\|\big)&<\frac{1}{2}
\end{align*}
where $C$, $\mu$, $\widehat{C}$ and $\kappa$ denote the constants from Proposition \ref{prop:more_reg}, Proposition \ref{prop:nonlinear}, Proposition \ref{prop:8} and Theorem \ref{thm:weak_non_con},
then there exists exactly one solution $\bar{y}=y(\bar{u})$ to \eqref{eq:state_eq_opt_con_prob} with minimal cost.
\end{proposition}

We now commence with the analysis which is required to obtain the existence result. 
 For the following results, for given $y$, we introduce the notation
 \begin{align}\label{eq:w_notat}
   \mathfrak{u}(s,x)=(U*\rho_{\mu y(s)})(x) = \int_{\mathbb R^d} U(x-\tilde{x}) \int_{\mathbb R^d} y(s,\tilde{x},v) \mu(v)\, \mathrm{d}v\,\mathrm{d}\tilde{x}, \\
   \mathfrak{v}(s,x)=(U*\rho_{\mu y(s)v})(x) = \int_{\mathbb R^d} U(x-\tilde{x}) \int_{\mathbb R^d} y(s,\tilde{x},v) \mu(v)v\, \mathrm{d}v\,\mathrm{d}\tilde{x}.
 \end{align}
Similarly, we use $\mathfrak{u}_n,\mathfrak{v}_n$ for a given $y_n$. First some continuity properties of the functions $h_i$ characterizing the nonlinearities are obtained. 

 \begin{proposition}\label{prop:weak_conv_hi}
   Let $\{y_n\}$ be bounded in $\mathcal{X}$ and let Assumption \ref{ass:U} hold. Then there exist ${y_{n_k}}$ and $y$ such that $y_{n_k}\rightharpoonup y \in  L^2(0,T;V_v)\cap W^{1,2}(0,T;Y_{-1})$ and $\int_0^t \langle h_i(y_{n_k}(s))- h_i(y(s)),\psi \rangle_{V_v',V_v} \,\mathrm{d}s \to 0$ for every $\psi \in V_v$ with $\psi(x,v)=0$ for $x\notin \Omega$, for arbitrary compact subsets $\Omega\subset \mathbb R^d,$ and every $t\in [0,T]$.
 \end{proposition}

 The proof of this proposition depends on the next lemma, whose proof is also provided in the Appendix.

 \begin{lemma}\label{lem:uv_strong_conv}
   Under the assumptions of Proposition \ref{prop:weak_conv_hi}, it holds that $\mathfrak{u}_n\to \mathfrak{u}$ in $L^2(0,t;C(\bar{\Omega}))$ and $\mathfrak{v}_{n,i}\to \mathfrak{v}_i$ in $L^2(0,t;C(\bar{\Omega})), i=1,\dots,d$ for every pre-compact subset $\Omega\subset \mathbb R^d$ and every $t\in [0,T]$.
 \end{lemma}

 \begin{proof}[Proof of Theorem \ref{thm:ex_opt_cont}]
 First we observe that by Proposition \ref{prop:nonlinear}, for $y_0$ sufficiently small $u\equiv 0$ is a feasible control. Let $\{u_n\} \in \mathcal{U}_{\mathrm{ad}}$ denote a minimizing sequence for \eqref{eq:opt_con_prob} with associated states $y_n=y(u_n)$.
 The feasibility of $u_n$ implies the boundedness of $u_n$ in $L^\infty(0,T)$. Thus $u_n\stackrel{*}{\rightharpoonup} \bar{u}$ for a subsequence denoted in the same manner and $\bar{u}$ satisfying $u^{\mathrm{min}}(t)\le\bar{u}(t)\le u^{\max}(t)$.  
 By the choice of $\mathcal{J}$, $y_n$ is bounded in $L^2(0,T;V_v)$ so that with Proposition \ref{prop:8}, $y_n$ is  bounded in $\mathcal{X}$ as well. Therefore there exist a weakly convergent subsequence, also denoted by $y_n$ and $\bar{y}\in \mathcal{X}$ such that $y_n\rightharpoonup \bar{y}$ in $L^2(0,T;V_v)\cap W^{1,2}(0,T;Y_{-1})$.
 
 Next, we argue that  $\bar{y}=y(\bar{u})$, the solution to \eqref{eq:state_eq_opt_con_prob} with $u=\bar{u}$. For this purpose, we pass to the limit in 
 \begin{equation}\label{eq:int_solution}
 \begin{aligned}
  \langle y_n(t)-y_0,\psi \rangle_{\mathcal{D}} &= \int_0^t  \langle y_n(s),(A+D)^* \psi \rangle _Y + \langle Ny_n(s)u_n(s)+Bu_n(s),\psi \rangle_{\mathcal{D}}  \\[1.2ex]
 & \quad  -\langle h_1(y_n(s))+ h_2(y_n(s)),\psi \rangle_{\mathcal{D}}   \,\, \mathrm{d}s
 \end{aligned}
 \end{equation}
 for every $t\in [0,T],$ and every  $\psi \in \mathcal{D}(A^*) $ with $\psi(x,v)=0$ 
  for $x \notin \Omega,$ and $\Omega  \subset \mathbb R^d$ an arbitrary compact subset. Since $y_n \rightharpoonup \bar{y}$ in $W^{1,2}(0,T;Y_{-1})$ it follows that $  \langle y_n(\cdot),\psi \rangle_{\mathcal{D}} \rightharpoonup 
 \langle \bar{y}(\cdot),\psi \rangle_{\mathcal{D}} 
$ in $W^{1,2}(0,T)$. By compactness of $W^{1,2}(0,T)$ in $C([0,T])$, we obtain
  \begin{align}\label{eq:strong_aux}
  \langle y_n(\cdot),\psi \rangle_{\mathcal{D}} \to 
 \langle \bar{y}(\cdot),\psi \rangle_{\mathcal{D}} \text{ in } C([0,T]).
 \end{align}
 Recall that $ N\in \mathcal{L}(Y,V_v')$. 
 By density of $\mathcal{D}(A^*)$ in $Y$, there exists $g_k\in \mathcal{D}(A^*)$ such that $g_k \to N'\psi$ in $Y$. We estimate
 \begin{align*}
 &| \langle \bar{y}, N'\psi \rangle_Y - \langle y_n,N'\psi \rangle |_{L^2(0,T)} \\
&\quad  \le  | \langle \bar{y},N'\psi - g_k\rangle_Y|_{L^2(0,T)}
  + | \langle \bar{y}-y_n,g_k\rangle_Y |_{L^2(0,T)}
+ | \langle y_n, g_k-N'\psi \rangle _Y |_{L^2(0,T)}.
 \end{align*}
 Taking first the limit w.r.t.~$k$ and subsequently w.r.t.~$n$ and using \eqref{eq:strong_aux} for the second term on the right hand side, we obtain 
 \begin{align}\label{eq:strong_aux2}
   \langle \bar{y}, N'\psi \rangle_Y  \to  \langle y_n,N'\psi \rangle \text { in } L^2(0,T).
 \end{align}
 Clearly this convergence also holds in $L^2(0,t)$ for every $t\in[0,T]$. Similarly, we have $y_n\rightharpoonup \bar{y}$ in $L^2(0,t;V_v)$ and $u_n\rightharpoonup\bar{u}$ in $L^2(0,t)$ for every $t \in [0,T]$. Utilizing \eqref{eq:strong_aux} and \eqref{eq:strong_aux2} and Proposition \ref{prop:weak_conv_hi}, we can pass to the limit in \eqref{eq:int_solution} to arrive at
 \begin{equation}\label{eq:int_solution_limit}
 \begin{aligned}
  \langle \bar{y}(t)-y_0,\psi \rangle_{\mathcal{D}} &= \int_0^t  \langle \bar{y}(s),(A+D)^* \psi \rangle _Y + \langle N\bar{y}(s)\bar{u}(s)+B\bar{u}(s),\psi \rangle_{\mathcal{D}}  \\[1.2ex]
 & \quad  -\langle h_1(\bar{y}(s))+ h_2(\bar{y}(s)),\psi \rangle_{\mathcal{D}}   \,\, \mathrm{d}s,
 \end{aligned}
 \end{equation}
for every $t\in [0,T],$ and every  $\psi \in \mathcal{D}(A^*) $ with $\psi(x,v)=0$ 
  for $x \notin \Omega,$ and $\Omega  \subset \mathbb R^d$ an arbitrary compact subset.
  
  Due to the fact that $C_0^\infty(\mathbb R^{2d})$ is a core for $\mathcal{D}(A^*)$ and $N'\in \mathcal{L}(\mathcal{D}(A^*),Y)$, a density argument implies that \eqref{eq:int_solution_limit} holds for all $\psi\in \mathcal{D}(A^*)$.
  
  Finally, since $\bar{y} \in L^2(0,T;V_v)\cap L^\infty(0,T;Y)$ and $\bar{u}\in L^2(0,T)$, we can apply Corollary \ref{cor:mild_sol_lin_sys} with $g=-(h_1(\bar{y})+h_2(\bar{y})) + N\bar{y}\bar{u}+B\bar{u} \in L^2(0,T;V_v')$ to obtain $\bar{y}\in C([0,T];Y)$. Thus $\bar{y}=y(\bar{u};u)$ is a solution corresponding to $\bar{u}$.
 
 Concerning optimality, we have the following estimates
 \begin{align*}
  \mathcal{J}(\bar{u})&=\frac{1}{2} \int_0^ T \| \bar{y}- y_{\dd}(s)\|_{V_v}^2 \, \mathrm{d}s+ \frac{\beta}{2}\int_0^T \bar{u}(s)^2 \, \mathrm{d}s \\
  &\le \liminf_{n\to\infty} \frac{1}{2} \int_0^ T \| y_n- y_{\dd}(s)\|_{V_v}^2 \, \mathrm{d}s+  \liminf_{n\to\infty} \frac{\beta}{2}\int_0^T u_n(s)^2 \, \mathrm{d}s \\
  &\le  \liminf_{n\to\infty}\left( \frac{1}{2} \int_0^ T \| y_n- y_{\dd}(s)\|_{V_v}^2 \, \mathrm{d}s+ \frac{\beta}{2}\int_0^T u_n(s)^2 \, \mathrm{d}s \right) =\inf_{u \in \mathcal{U}_{\mathrm{ad}}} \mathcal{J}(u).
 \end{align*}
 Consequently, $\bar{u}$ is an optimal control for \eqref{eq:opt_con_prob}.
\end{proof}
 
 \begin{proof}[Proof of Proposition \ref{prop:unique}]
From Theorem \ref{thm:ex_opt_cont}, it follows that there exists at least one optimal control $\bar{u}$ to \eqref{eq:opt_con_prob}. Assume there were two solutions $y_1(\bar{u})$, $y_2(\bar{u})$ with minimal cost $\mathcal{J}_{\min}$. Since $\tilde{u}\equiv 0$ is a feasible control, for $i=1,2$ it follows that
\begin{align*}
 \frac{1}{2} \|y_i(\bar{u})-y_{\dd}\|_{L^2(0,T;V_v)}^2 + \frac{\beta}{2} \|\bar{u}\|_{L^2(0,T)} \le \frac{1}{2} \|y(\tilde{u})-y_{\dd}\|_{L^2(0,T;V_v)}^2
\end{align*}
so that $\|y_i-y_{\dd}\|_{L^2(0,T;V_v)}\le \|y(\tilde{u})-y_{\dd}\|_{L^2(0,T;V_v)}$. By Theorem \ref{thm:weak_non_con} we therefore also obtain
\begin{equation}\label{eq:aux5}
\begin{aligned}
 \|y_i(\bar{u})\|_{L^2(0,T;V_v)}&\le  \|y_i(\bar{u})-y_{\dd}\|_{L^2(0,T;V_v)} +\| y_{\dd}\|_{L^2(0,T;V_v)}  \le \kappa + 2 \| y_{\dd}\|_{L^2(0,T;V_v)}
\end{aligned}
\end{equation}
for $i=1,2$. 
By Proposition \ref{prop:more_reg}  for $i\in \{1,2\}$
\begin{equation}\label{eq:aux3}
\VERT y_i(\bar{u}) \VERT \le \widehat{C}(\|y_0\|_Y + \|\tilde g\|_{L^2(0,T;V'_v}),
\end{equation}
holds, where  $\tilde g= - h_1 +h_2 + u N y_i,$  and hence by Lemma \ref{lem:loc_lipschitz_stat}
\begin{equation*}
\begin{aligned}
\|h_1(y_i)\|_{L^2(0,T;V_v')} &\le \|U\|_{L^2(\mathbb R^d)}\|y_i\|_{C([0,T];Y)} \|y_i\|_{L^2(0,T;V_v)}\\
&\le \sqrt{d} \|U\|_{L^2(\mathbb R^d)}\|y_i\|_{C([0,T];Y)} \|y_i\|_{L^2(0,T;V_v)}\\[1.5ex]
\|h_2(y_i)\|_{L^2(0,T;V_v')} &\le \sqrt{d}\|U\|_{L^2(\mathbb R^d)}\|y_i\|_{C([0,T];Y)} \|y_i\|_{L^2(0,T;Y)}\\
&\le \sqrt{d}\|U\|_{L^2(\mathbb R^d)}\|y_i\|_{C([0,T];Y)} \|y_i\|_{L^2(0,T;V_v)}. 
\end{aligned}
\end{equation*}
Thus by Proposition \ref{prop:8},  for $i,j \in \{1,2\}$, we have
\begin{align*}
\|h_j(y_i)\|_{L^2(0,T;V_v')} &\le \sqrt{d}\|U\|_{L^2(\mathbb R^d)} C(\|y_0\|,u^\infty,\kappa +2\|y_{\dd}\|) \|y_i\|_{L^2(0,T;V_v)}.
\end{align*}
From \eqref{eq:aux3} and $\|y_i(\bar{u})\|_{L^2(0,T;V_v)}\le \kappa + 2 \| y_{\dd}\|_{L^2(0,T;V_v)}$ it thus follows
\begin{equation}\label{eq:aux4}
\begin{aligned}
 \VERT y_i(\bar{u}) \VERT& \le \widehat{C}\big (\|y_0\|+2\sqrt{d}\|U\|_{L^2(\mathbb R^d)}C(\|y_0\|,u^\infty,\kappa +2\|y_{\dd}\|) \|y_i\|_{L^2(0,T;V_v)} \\
 &\quad + u^\infty \|N\|_{\mathcal{L}(Y,V'_v)} \|y_i\|_{L^2(0,T;V_v)} \big) \\
 &\le \widehat{C}\big (\|y_0\| +(\kappa + 2 \| y_{\dd}\|_{L^2(0,T;V_v)})(2\sqrt{d}\|U\|_{L^2(\mathbb R^d)}C(\|y_0\|,u^\infty,\kappa +2\|y_{\dd}\|)\\ &\quad + u^\infty \|N\|_{\mathcal{L}(Y,V'_v)}) \big)
 \end{aligned}
\end{equation}
where in the last inequality we used \eqref{eq:aux5}.
We now set $e=y_1-y_2$ and observe that $e$ satisfies
\begin{equation*}
\dot{e}(t)=Ae(t) + De(t) + g(t), \quad e(0)=0, 
 \end{equation*}
with 
\begin{equation*}
g(t)=-(h_1(y_1(t))-h_1(y_2(t))) + h_2(y_1(t)) - h_2(y_2(t)) + \bar{u}(t) N (y_1(t)-y_2(t)). 
\end{equation*}
By Lemma \ref{lem:loc_lipschitz} we find 
\begin{equation*}
\|g\|_{L^2(0,T;V'_v)} \le 2\sqrt{d}\|U\|_{L^2(\mathbb R^d)} (\VERT y_1\VERT+ \VERT y_2\VERT) (\VERT e\VERT + \|N\|_{\mathcal{L}(Y,V_v')} u^\infty \|e\|_{L^2(0,T;Y)}).
\end{equation*}
Proposition \ref{prop:more_reg} implies that 
\begin{equation*}
\begin{array}{ll}
\VERT e \VERT    \le \widehat{C}\|g\|_{L^2(0,T:V'_v)}  \le 2\sqrt{d}\|U\|_{L^2(\mathbb R^d)}\widehat{C}( \VERT y_1\VERT + \VERT y_2 \VERT) \VERT e \VERT (1+ \|N\|_{\mathcal{L}(Y,V_v')} u^\infty).
\end{array}
\end{equation*}
Combining this estimate with \eqref{eq:aux4} yields
\begin{align*}
 \VERT e \VERT   & \le 2 \sqrt{d} \|U\|_{L^2(\mathbb R^d)}\widehat{C} \VERT e \VERT (1+ \|N\|_{\mathcal{L}(Y,V_v')} )\cdot \\
 & \quad  2\widehat{C} \bigg(\|y_0\| +(\kappa + 2 \| y_{\dd}\|_{L^2(0,T;V_v)})\big(2\sqrt{d}\|U\|_{L^2(\mathbb R^d)}C(\|y_0\|,u^\infty,\kappa +2\|y_{\dd}\|) + u^\infty \|N\|_{\mathcal{L}(Y,V'_v)}\big) \bigg) 
\end{align*}
where we used  $u^\infty\le 1$. This,   together with the assumptions
\begin{align*}
 8d\widehat{C}^2\|U\|_{L^2(\mathbb R^d)}^2 (1+\|N\| )(2\|y_{\mathrm{d}}\|+\kappa) C(\|y_0\|,u^\infty,\kappa+2\|y_{\mathrm{d}}\|)&\le \frac{1}{2} \\
 4\sqrt{d}\widehat{C}^2\|U\|_{L^2(\mathbb R^d)} (1+\|N\| )\big(\|y_0\|+(2\|y_{\mathrm{d}}\|+\kappa)u^\infty\|N\|\big)&<\frac{1}{2}
\end{align*}
implies the uniqueness claim. 
\end{proof}

\section{Conclusion}

We have analyzed local existence of solutions to a specific nonlinear and nonlocal kinetic Fokker-Planck equation. Due to a lack of coercivity of the underlying operators, standard (variational) solution techniques were not directly applicable and we instead resorted to the concept of admissible control operators and suitable Lipschitz estimates for the nonlinearities which were utilized in a fixed point technique. We subsequently introduced and analyzed a quadratic tracking type cost functionals for which we showed the existence of a local solution and the uniqueness of its associated state.
 
Several questions remain open and deserve a further detailed analysis. For example, while the existence of optimal controls obviously calls for first order necessary optimality conditions, a corresponding sensitivity analysis appears to be far from straightforward as the uniqueness of solutions (other than the ones corresponding to the optimal control) is not clear at this point. Furthermore, a study of the optimal control problem on an infinite-horizon or the construction of feedback controls seem interesting follow up research questions. Similarly, discretization strategies both of the uncontrolled equation as well as the optimal control problem apparently  have  not received particular attention for hypocoercive problems and could be the focus of future work. It would further be of interest whether Assumption \ref{ass:U} concerning  extra regularity of the potential $U$ could be overcome by the velocity averaging effect of the first order hyperbolic operator $\partial_t y  + v  \cdot \nabla_x y $, see, e.g., \cite{DiLi89}. This would be the case if the nonlocal operators $\rho_{\mu y}$ and  $\rho_{\mu y v}$ had local support in the $v$ variable.
  
\section*{Appendix}

\begin{proof}[Proof of Lemma \ref{lem:loc_lipschitz_stat}]
   
Assume that $w,\tilde{w}\in V_v$. We consider $U*\rho_{\mu w} R_0\tilde{w}$ and begin with an estimate for $U*\rho_{\mu w}$. Young's inequality for convolutions yields
  \begin{align*}
   \|U*\rho_{\mu w}\|_{L^\infty(\mathbb R^d)} \le \|U\|_{L^2(\mathbb R^d)} \|\rho_{\mu w}\|_{L^2(\mathbb R^d)}.
  \end{align*}
 We further have that
 \begin{align*}
  \|\rho_{\mu w}\|_{L^2(\mathbb R^d)} &= \left( \int_{\mathbb R^d} \left|\rho_{\mu w}(x)\right|^2 \mathrm{d}x \right)^{\frac{1}{2}}
  = \left( \int_{\mathbb R^d} \left|\int_{\mathbb R^d} \mu(v)w(x,v)\,\mathrm{d}v\right|^2 \mathrm{d}x \right)^{\frac{1}{2}}.
  \end{align*}
With the Minkowski integral inequality, we then obtain
\begin{align*}
  \|\rho_{\mu w}\|_{L^2(\mathbb R^d)} &\le \int_{\mathbb R^d} \left(\int_{\mathbb R^d} \left|\mu(v)w(x,v) \right|^2 \mathrm{d}x \right)^{\frac{1}{2}} \mathrm{d}v  = \int_{\mathbb R^d} \mu(v)^{\frac{1}{2}} \left(\int_{\mathbb R^d} \mu(v)\left|w(x,v) \right|^2 \mathrm{d}x \right)^{\frac{1}{2}} \mathrm{d}v.
\end{align*}
Applying Cauchy-Schwarz w.r.t.~the $v$ variable shows that
\begin{align}\label{eq:aux6}
  \|\rho_{\mu w(t)}\|_{L^2(\mathbb R^d)} &\le \underbrace{\left(\int_{\mathbb R^d} \mu(v)\,\mathrm{d}v\right)^{\frac{1}{2}}}_{=1} \underbrace{\left(\int_{\mathbb R^d}\int_{\mathbb R^d}\mu(v)|w(x,v)|^2\,\mathrm{d}x\,\mathrm{d}v \right)^{\frac{1}{2}}}_{=\|w\|_{Y}}.
\end{align}
Up to this point, we conclude that
\begin{align}\label{eq:aux1}
 \|U*\rho_{\mu w}\|_{L^\infty(\mathbb R^d)} \le \|U\|_{L^2(\mathbb R^d)} \|w\|_Y
\end{align}
which then implies that
\begin{align*}
 \|U*\rho_{\mu w} R_0 \tilde{w} \|_{V_v'} &=  \sup_{\|\psi\|_{V_v}=1}\left| \int_{\mathbb R^d}\int_{\mathbb R^d} \mu(v) (U*\rho_{\mu w})(x) R_0 \tilde{w}(x,v) \psi(x,v)\,  \mathrm{d}x\, \mathrm{d}v \right|  \\
 &\le \|U*\rho_{\mu w} \|_{L^\infty(\mathbb R^d)}  \sup_{\|\psi\|_{V_v}=1} |\langle R_0\tilde{w},\psi \rangle _{Y}|   \\
 &\le \|U\|_{L^2(\mathbb R^d)} \|w \|_{Y}\sup_{\|\psi\|_{V_v}=1} |\langle R_0\tilde{w},\psi \rangle _{Y}|.
\end{align*}
For almost all $\psi \in C_0^\infty(\mathbb R^{2d})$, let us note that with \eqref{eq:mu_aux} we have
\begin{align*}
  \langle R_0\tilde{w},\psi\rangle_Y &= \int_{\mathbb R^{2d}} \mu \psi (v\cdot \nabla_v \tilde{w} - \Delta_v \tilde{w}  ) \,\mathrm{d}x\,\mathrm{d}v \\
  &= \int_{\mathbb R^{2d}} \mu \psi v\cdot \nabla_v \tilde{w}  + \nabla_v  \tilde{w}\cdot \nabla_v (\mu \psi)  \,\mathrm{d}x\,\mathrm{d}v \\
  &= \int_{\mathbb R^{2d}} \mu \psi v\cdot \nabla_v \tilde{w}  + \mu \nabla_v  \tilde{w}\cdot \nabla_v \psi + \psi \nabla_v \tilde{w} \cdot \nabla_v \mu  \,\mathrm{d}x\,\mathrm{d}v \\
  &= \int_{\mathbb R^{2d}} \mu \nabla_v  \tilde{w}\cdot \nabla_v \psi\,\mathrm{d}x\,\mathrm{d}v = \langle \nabla_v \tilde{w},\nabla_v\psi \rangle _{Y^{2d}}.
\end{align*}
By density of $C^\infty_0(\R^{2d})$ in $V_v$, which can be argued by classical approximation techniques, see for instance \cite[Section 3]{Ada75}, the above equality holds for all $\psi \in V_v$.
We finally arrive at
\begin{align*}
 \|U*\rho_{\mu w} R_0 \tilde{w} \|_{V_v'} 
  &\le \|U\|_{L^2(\mathbb R^d)} \|w \|_{Y} \sup_{\|\psi\|_{V_v}=1} |\langle R_0\tilde{w}(t),\psi \rangle _{Y}|\\
 &\le \|U\|_{L^2(\mathbb R^d)}  \|w \|_{Y} \sup_{\|\psi\|_{V_v}=1} \|\nabla_v\tilde{w}(t)\|_{Y^{2d}}  \|\nabla_v\psi\|_{Y^{2d}}    \\
 & \le  \|U\|_{L^2(\mathbb R^d)}  \|w \|_{Y} \|\nabla_v \tilde{w}\|_{Y} .
\end{align*}
The first assertion now follows from
\begin{align*}
 h_1(y)-h_1(z)&= U*\rho_{\mu y} R_0y-U*\rho_{\mu z} R_0z \\
 &=U*\rho_{\mu y} R_0y-U*\rho_{\mu y} R_0z+U*\rho_{\mu y} R_0z-U*\rho_{\mu z} R_0z \\
 &=U*\rho_{\mu y} R_0(y-z)+U*\rho_{\mu (y-z)}  R_0z.
\end{align*}
Let us next turn to $U*\rho_{v\mu w}\cdot ( \nabla_v \tilde{w}-\tilde{w}v)$. 
We observe that with Young's inequality for convolutions and \eqref{eq:derive_y_aux5a},  it holds that
\begin{align}\label{eq:derive_y_aux5}
 \| U*\rho_{v \mu w}\|_{(L^\infty(\mathbb R^d))^d} \le \| U \|_{L^2(\mathbb R^d)} \| \rho_{v \mu w} \|_{(L^2(\mathbb R^d))^d}\le \sqrt{d}\| U \|_{L^2(\mathbb R^d)} \| w\|_{Y}.
\end{align}
Before we continue, recall that with \eqref{eq:mu_aux} we have the identity
\begin{align*}
 \mu^{-1}\nabla_v(\mu \tilde{w})&= \mu^{-1}( \tilde{w}\nabla_v \mu +\mu \nabla_v \tilde{w})  = \mu^{-1}( -\tilde{w}v\mu +\mu \nabla_v \tilde{w})= \nabla _v \tilde{w}-\tilde{w}v.
\end{align*}
Hence, we obtain
\begin{align*}
& \| U*\rho_{v \mu w} \cdot (\nabla_v \tilde{w}-\tilde{w}v) \|_{V_v'}\\
&\ \ = \sup_{\|\psi\|_{V_v}=1}\left| \int_{\mathbb R^d}\int_{\mathbb R^d} \mu(v) (U*\rho_{v\mu w})(x) \cdot (\nabla_v \tilde{w}(x,v)-\tilde{w}(x,v)v) \psi(x,v)\,  \mathrm{d}x\, \mathrm{d}v \right| \\
&\ \ = \sup_{\|\psi\|_{V_v}=1}\left| \int_{\mathbb R^d}\int_{\mathbb R^d}   (U*\rho_{v\mu w})(x) \cdot \nabla_v(\mu(v) \tilde{w}(x,v)) \psi(x,v)\,  \mathrm{d}x\, \mathrm{d}v \right| .
\end{align*}
Using \eqref{eq:derive_y_aux5} we find that
\begin{align*}
& \| U*\rho_{v \mu w} \cdot (\nabla_v \tilde{w}-\tilde{w}v) \|_{V_v'}\\
& \ \ \le  \| U*\rho_{v \mu w} \|_{L^\infty(\mathbb R^d)} \sup_{\|\psi\|_{V_v}=1}\left\| \int_{\mathbb R^d}\int_{\mathbb R^d}    \nabla_v(\mu(v) \tilde{w}(x,v)) \psi(x,v)\,  \mathrm{d}x\, \mathrm{d}v \right\|_{\mathbb R^d}  \\
& \ \ \le \sqrt{d} \| U \|_{L^2(\mathbb R^d)} \| w \|_{Y} \sup_{\|\psi\|_{V_v}=1}\left\| \int_{\mathbb R^d}\int_{\mathbb R^d} \mu(v) \tilde{w}(x,v) \nabla_v\psi(x,v)\,  \mathrm{d}x\, \mathrm{d}v \right\|_{\mathbb R^d} \\
& \ \ \le \sqrt{d} \| U \|_{L^2(\mathbb R^d)} \| w \|_{Y} \sup_{\|\psi\|_{V_v}=1} \int_{\mathbb R^d}\int_{\mathbb R^d} |\mu(v)^{\frac{1}{2}} \tilde{w}(x,v)| \cdot \| \mu^{\frac{1}{2}}(v)\nabla_v\psi(x,v)\|_{\mathbb R^d} \,  \mathrm{d}x\, \mathrm{d}v .
\end{align*}
Applying Cauchy Schwarz on $L^2(\mathbb R^{2d})$ allows to conclude that
\begin{align*}
 & \| U*\rho_{v \mu w} \cdot (\nabla_v \tilde{w}-\tilde{w}v) \|_{V_v'}\\
& \ \ \le \sqrt{d} \| U \|_{L^2(\mathbb R^d)} \| w \|_{Y} \sup_{\|\psi\|_{V_v}=1} \underbrace{\| \mu^{\frac{1}{2}} \tilde{w}\|_{L^2(\mathbb R^{2d})}}_{=\|\tilde{w}\|_Y}
\underbrace{
\| \mu^{\frac{1}{2}} \nabla_v\psi \|_{(L^2(\mathbb R^{2d}))^d}}_{=\|\nabla_v \psi\|_{Y^d}}  \\
& \ \ \le \sqrt{d} \| U \|_{L^2(\mathbb R^d)} \| w \|_{Y}  \|\tilde{w}\|_Y .
\end{align*}
Finally, we obtain the second assertion since
\begin{align*}
 h_2(y)-h_2(z) &= U*\rho_{\mu yv} \cdot (\nabla_v y-yv) - U*\rho_{\mu zv} \cdot (\nabla_v z-zv) \\
 & = U*\rho_{\mu yv} \cdot (\nabla_v y-yv) - U*\rho_{\mu yv} \cdot (\nabla_v z-zv) \\
 &\quad +U*\rho_{\mu yv} \cdot (\nabla_v z-zv) - U*\rho_{\mu zv} \cdot (\nabla_v z-zv) \\
 &=U*\rho_{\mu yv} \cdot (\nabla_v (y-z)-(y-z)v) +U*\rho_{\mu (y-z)v} \cdot (\nabla_v z-zv).
\end{align*}
\end{proof}

\begin{proof}[Proof of Proposition \ref{prop:weak_conv_hi}]
The assumed boundedness of $\{y_n\}$ implies the existence of a weakly convergent subsequence $y_{n_k} \rightharpoonup y$ in  $\mathcal{X}$. Subsequently, the second subscript $k$ will be dropped. 

\emph{Step 1.}
We first consider $h_1$. Let $t>0$ and $\psi \in V_v$ be as announced. It then holds that
   \begin{align*}
    &\int_0^t \langle h_1(y_n(s))-h_1(y(s)),\psi \rangle_{V_v',V_v} \,\mathrm{d}s \\
    &\quad = \int _ 0^t \langle U*\rho_{\mu (y_n(s)-y(s))}\overline{R}y_n(s),\psi \rangle_{V_v',V_v}\,\mathrm{d}s   +
    \int _ 0^t \langle U*\rho_{\mu y(s)}\overline{R}(y_n(s)-y(s)),\psi \rangle_{V_v',V_v}\,\mathrm{d}s \\
    &\quad = \underbrace{\int _ 0^t \langle (\mathfrak{u}_n(s)-\mathfrak{u}(s))\overline{R}y_n(s),\psi \rangle_{V_v',V_v}\,\mathrm{d}s}_{=\mathrm{I}} +  \underbrace{
    \int _ 0^t \langle \overline{R}(y_n(s)-y(s)),\mathfrak{u}(s)\psi \rangle_{V_v',V_v}\,\mathrm{d}s}_{=\mathrm{II}}.
   \end{align*}

   We already know from \eqref{eq:aux1} that 
   \begin{align*}
    \| \mathfrak{u}_n(s,\cdot)\|_{L^\infty(\mathbb R^d)} \le \|U\|_{L^2(\mathbb R^d)} \| y_n(s)\|_Y
   \end{align*}
   for all $s\in [0,T]$. Since $\{y_n\}$ is bounded in $L^\infty(0,T;Y)$, we can find a constant $C$ such that
   $
    |\mathfrak{u}_n(s,x)| \le C
   $
   for all $(s,x)\in [0,T]\times \mathbb R^d$. Analogously we obtain
   $ |\mathfrak{u}(s,x)| \le C
   $
   for all $(s,x)\in [0,T]\times \mathbb R^d$.

   Now we turn to $\mathrm{II}$. First, we observe that $\mathfrak{u}\psi \in L^2(0,T;V_v)$. Since $y_n\rightharpoonup y$ in $L^2(0,T;V_v)$, it follows that $\overline{R}y_n \rightharpoonup \overline{R}y$ in $L^2(0,T;V_v')$ and consequently $\mathrm{II}$ tends to $0$ as $n\to\infty$.

   Next, we turn to $\mathrm{I}$ and observe that for $\psi \in L^2(0,T;V_v)$ with $\psi(s,x,v)=0$ for $x\notin \Omega$, it holds that
  \begin{align*}
 &  \int _ 0^t \langle (\mathfrak{u}_n(s)-\mathfrak{u}(s))\overline{R}y_n(s),\psi \rangle_{V_v',V_v}\,\mathrm{d}s\\ 
 &\quad =
   \int _ 0^t \int_{\mathbb R^{d}} (\mathfrak{u}_n(s,x)-\mathfrak{u}(s,x))\int_{\mathbb R^{d}}\psi(x,v) \overline{R}y_n(s,x,v) \mu(v)\,\mathrm{d}x\,\mathrm{d}v\,\mathrm{d}s \\
 &\quad =
   \int _ 0^t \int_{\Omega} (\mathfrak{u}_n(s,x)-\mathfrak{u}(s,x))\int_{\mathbb R^{d}}\psi(x,v) \overline{R}y_n(s,x,v) \mu(v)\,\mathrm{d}v\,\mathrm{d}x\,\mathrm{d}s\\
 &\quad \le
   \int _ 0^t  \|\mathfrak{u}_n(s)-\mathfrak{u}(s)\|_{C(\bar{\Omega})}\int_{\mathbb R^d}\int_{\mathbb R^{d}}|\psi(x,v) \overline{R}y_n(s,x,v) \mu(v)|\,\mathrm{d}v\,\mathrm{d}x\,\mathrm{d}s\\
 &\quad \le
   \int _ 0^t  \|\mathfrak{u}_n(s)-\mathfrak{u}(s)\|_{C(\bar{\Omega})} |\langle  \overline{R}y_n(s),\psi \rangle |_{V_v',V_v} \,\mathrm{d}s \\
  &\quad \le  \|\mathfrak{u}_n-\mathfrak{u}\|_{L^2(0,T;C(\bar{\Omega}))} \|\psi \| _{V_v} \| \overline{R}y_n\|_{L^2(0,T;V_v')}.
  \end{align*}
  Since from Lemma \ref{lem:uv_strong_conv}  we have $\mathfrak{u}_n\rightarrow \mathfrak{u}$ in $L^2(0,T;C(\bar{\Omega}))$
   and  $\overline{R}y_n $ is bounded in $L^2(0,T;V_v')$, $\mathrm{I}$ tends to $0$ as $n\to\infty$.

\emph{Step 2.} Next we consider $h_2$.
From \eqref{eq:derive_y_aux5} we conclude that
\begin{align*}
\|\mathfrak{v}_n(s,\cdot)\|_{(L^\infty(\mathbb R^d))^d} \le C \| y_n(s)\|_Y
\end{align*}
which, together with the boundedness of $\{y_n\}$ in $L^\infty(0,T;Y)$ ensures the existence of a constant $C$ such that $
\| \mathfrak{v}_{n,i}(s,x)\|\le C  $ for almost all $(s,x)\in [0,T]\times \mathbb R^d$ and $i=1,\dots,d$. An analogous estimate holds true for $\mathfrak{v}_{i}, i=1,\dots,d$.

For what follows, we introduce the notation $
Sy=\nabla_v y-yv$. With derivations similar to the ones in the proof of Lemma \ref{lem:loc_lipschitz}, one can show that $(Sz)_i\in \mathcal{L}(Y,V_v')$. Indeed, observe that
\begin{align*}
  \|e_i \cdot Sz \|_{V_v'}&=\sup_{\|\psi\|_{V_v}=1}| \langle e_i\cdot Sz, \psi \rangle_{V_v',V_v}|\\
  &=\sup_{\|\psi\|_{V_v}=1}\left|\int_{\mathbb R^d}\int_{\mathbb R^d} e_i\cdot \nabla_v (\mu(v)z(x,v))\psi(x,v)\,\mathrm{d}x\,\mathrm{d}v \right| \\
  &=\sup_{\|\psi\|_{V_v}=1}\left|\int_{\mathbb R^d}\int_{\mathbb R^d} \mu(v)z(x,v)\psi_{v_i}(x,v)\,\mathrm{d}x\,\mathrm{d}v \right| \\
  &\le \sup_{\|\psi\|_{V_v}=1}\int_{\mathbb R^d}\int_{\mathbb R^d} |\mu(v)^{\frac{1}{2}}z(x,v)|\cdot |\psi_{v_i}(x,v)\mu(v)^{\frac{1}{2}} |\,\mathrm{d}x\,\mathrm{d}v \\
  &\le \sup_{\|\psi\|_{V_v}=1}\|z\|_Y \cdot \|\nabla_v\psi\|_{  Y^d} \le \| z \|_Y.
\end{align*} 
With these preparations, we obtain for $\psi \in V_v$
\begin{align*}
    &\int_0^t \langle h_2(y_n(s))-h_2(y(s)),\psi \rangle_{V_v',V_v} \,\mathrm{d}s \\
    &\quad = \int _ 0^t \langle U*\rho_{\mu v (y_n(s)-y(s))}\cdot Sy(s),\psi \rangle_{V_v',V_v}\,\mathrm{d}s   +
    \int _ 0^t \langle U*\rho_{\mu v y(s)}\cdot S(y_n(s)-y(s)),\psi \rangle_{V_v',V_v}\,\mathrm{d}s\\
    &\quad = \underbrace{\int _ 0^t \langle (\mathfrak{v}_n(s)-\mathfrak{v}(s))\cdot Sy_n(s),\psi \rangle_{V_v',V_v}\,\mathrm{d}s}_{=\mathrm{I}} +  \underbrace{
    \int _ 0^t \langle \mathfrak{v}(s)\cdot S(y_n(s)-y(s)),\psi \rangle_{V_v',V_v}\,\mathrm{d}s}_{=\mathrm{II}}.
   \end{align*}
  
   Let us focus on $\mathrm{II}$ first. Note that we have 
   \begin{align*}
     \int_0^t \langle \mathfrak{v}(s)\cdot S(y_n(s)-y(s)),\psi \rangle_{V_v',V_v}\,\mathrm{d}s&=\sum_{i=1}^d \int_0^t 
     \langle \mathfrak{v}_i(s) e_i \cdot S(y_n(s)-y(s)),\psi \rangle_{V_v',V_v} \,\mathrm{d}s \\
     & =\sum_{i=1}^d \int_0^t 
     \langle e_i \cdot S(y_n(s)-y(s)),\mathfrak{v}_i(s)\psi \rangle_{V_v',V_v} \,\mathrm{d}s.
   \end{align*}
   Since $\mathfrak{v}_i\psi \in L^2(0,T;V_v), y_n\rightharpoonup y \in L^2(0,T;V_v)$ and $e_i\cdot S \in \mathcal{L}(Y,V_v')$, we conclude that $\mathrm{II}$ tends to $0$ as $n\to \infty$. 
   
   Let us turn to $\mathrm{I}$. It holds that
   \begin{align*}
    &\int _ 0^t \langle (\mathfrak{v}_n(s)-\mathfrak{v}(s))\cdot Sy_n(s),\psi \rangle_{V_v',V_v}\,\mathrm{d}s \\
    &\quad= \sum_{i=1}^d \int _ 0^t \langle (\mathfrak{v}_{n,i}(s)-\mathfrak{v}_i(s)) e_i\cdot Sy_n(s),\psi \rangle_{V_v',V_v}\,\mathrm{d}s.
   \end{align*}
   For $i=1,\dots,d$, using $\psi(x,v)=0$ for $x\notin \Omega$, we find that 
  \begin{align*}
 &  \int _ 0^t \langle (\mathfrak{v}_{n,i}(s)-\mathfrak{v}_i(s)) e_i\cdot Sy_n(s),\psi \rangle_{V_v',V_v}\,\mathrm{d}s\\ 
 &\quad =
   \int _ 0^t \int_{\mathbb R^{d}} (\mathfrak{v}_{n,i}(s,x)-\mathfrak{v}_i(s,x))\int_{\mathbb R^{d}}\psi(x,v) e_i\cdot Sy_n(s,x,v) \mu(v)\,\mathrm{d}x\,\mathrm{d}v\,\mathrm{d}s \\
 &\quad =
   \int _ 0^t \int_{\Omega} (\mathfrak{v}_{n,i}(s,x)-\mathfrak{v}_i(s,x))\int_{\mathbb R^{d}}\psi(x,v) e_i\cdot Sy_n(s,x,v) \mu(v)\,\mathrm{d}v\,\mathrm{d}x\,\mathrm{d}s\\
 &\quad \le
   \int _ 0^t  \|\mathfrak{v}_{n,i}(s)-\mathfrak{v}_i(s)\|_{C(\bar{\Omega})}\int_{\mathbb R^d}\int_{\mathbb R^{d}}|\psi(x,v) e_i\cdot Sy_n(s,x,v) \mu(v)|\,\mathrm{d}v\,\mathrm{d}x\,\mathrm{d}s\\
 &\quad \le
   \int _ 0^t  \|\mathfrak{v}_{n,i}(s)-\mathfrak{v}_i(s)\|_{C(\bar{\Omega})} |\langle  e_i\cdot Sy_n(s),\psi \rangle |_{V_v',V_v} \,\mathrm{d}s \\
  &\quad \le  \|\mathfrak{v}_{n,i}-\mathfrak{v}_i\|_{L^2(0,T;C(\bar{\Omega}))} \|\psi \| _{V_v} \| e_i\cdot Sy_n\|_{L^2(0,T;V_v')}.
  \end{align*}
   From Lemma \ref{lem:uv_strong_conv}, we know that  
   $\mathfrak{v}_{n,i}\rightarrow \mathfrak{v}$ in $L^2(0,T;C(\bar{\Omega}))$.  This completes the proof. 
\end{proof}

 \begin{proof}[Proof of Lemma \ref{lem:uv_strong_conv}]
At first we show that $\{\mathfrak{v}_{n,i}\}$ is bounded in $W^{1,2}(0,T;L^2(\mathbb R^d))$. Indeed we estimate
  \begin{align*}
   &\|\mathfrak{v}_{n,i}\|_{L^2(0,T;L^2(\mathbb R^d))}^2 = \int_0^T \int_{\mathbb R^d}
   \left| \int_{\mathbb R^d} U(x-\tilde{x})\int_{\mathbb R^d}y_n(t,\tilde{x},v)\mu(v)v_i\,\mathrm{d}v \, \mathrm{d}\tilde{x}\right|^2  \mathrm{d}x \,\mathrm{d}t \\
  &\quad =\int_0^T \left\| \left( U * \int_{\mathbb R^d}y_n(t,v)\mu(v)v_i\,\mathrm{d}v\right)(\cdot) \right\|_{L^2(\mathbb R^d)}^2 \,\mathrm{d}t \\
  &\quad\le \| U\|_{L^1(\mathbb R^d)}^2 \int_0^T \int_{\mathbb R^d} \left| \int_{\mathbb R^d} y_n(t,x,v)\mu^{\frac{1}{2}}(v)\mu^{\frac{1}{2}}(v)v_i\,\mathrm{d}v\right|^2\mathrm{d}x\,\mathrm{d}t \\
  &\quad\le \| U\|_{L^1(\mathbb R^d)}^2 \int_0^T \int_{\mathbb R^d} \left( \int_{\mathbb R^d} y_n^2(t,x,v)\mu(v)\,\mathrm{d}v \right)
  \left( \int_{\mathbb R^d} \mu(v)v_i^2\,\mathrm{d}v\right) \mathrm{d}x\,\mathrm{d}t\\
  &\quad\le \| U\|_{L^1(\mathbb R^d)}^2\|\mu v_i^2\|_{L^1(\mathbb R^d)} \int_0^T \int_{\mathbb R^d}  \int_{\mathbb R^d} y_n^2(t,x,v)\mu(v)\,\mathrm{d}v \, \mathrm{d}x\,\mathrm{d}t \\
  &\quad= \| U\|_{L^1(\mathbb R^d)}^2\|\mu v_i^2\|_{L^1(\mathbb R^d)} \|y_n\|_{L^2(0,T;Y)}^2.
  \end{align*}
  If $U\in W^{s,1}(\mathbb R^d)$, the previous computations can be repeated, using differentiation of convolutions, to obtain
  \begin{align*}
   \|\mathfrak{v}_{n,i}\|_{L^2(0,T;H^s(\mathbb R^d))}^2 &\le \| U\|^2_{W^{s,1}(\mathbb R^d)}\|\mu v_i^2\|_{L^1(\mathbb R^d)} \|y_n\|_{L^2(0,T;Y)}^2.
  \end{align*}
  For the temporal derivative of $\mathfrak{v}_{n,i}$, we continue with
  \begin{align*}
   &\|\dot{\mathfrak{v}}_{n,i}\|_{L^2(0,T;L^2(\Omega))}^2 =    \int_0^T \int_{\Omega} \left| \int_{\mathbb R^{2d}} U(x-\tilde{x}) \dot{y}_n(t,\tilde{x},v) \mu(v)v_i\,\mathrm{d} v\,\mathrm{d} \tilde{x} \right|^2\,\mathrm{d}x\,\mathrm{d}t \\
   &\quad= \int_0^T \int_{\Omega} \left| \langle \dot{y}_n(t), (\cdot)_i\cdot U(x-\cdot) \rangle_{\mathcal{D}} \right|^2\,\mathrm{d}x\,\mathrm{d}t   \\
   &\quad\le \int_0^T \int_{\Omega} \left| \sup\limits_{\phi \in \mathcal{D}(A^*),\|\phi\|_{\mathcal{D}(A^*)}\le 1} \langle \dot{y}_n(t),\phi \rangle_{\mathcal{D}}\right|^2 \|(\cdot)_i\cdot U(x-\cdot) \|^2_{\mathcal{D}(A^*)} \,\mathrm{d}x\,\mathrm{d}t \\
   &\quad\le \int_0^T \int_{\Omega} \| \dot{y}_n(t) \|_{Y_{-1}}^2\|(\cdot)_i\cdot U(x-\cdot) \|^2_{\mathcal{D}(A^*)}\,\mathrm{d}x\,\mathrm{d}t \\
   &\quad=|\Omega| \| \dot{y}_n\|_{L^2(0,T;Y_{-1})}^2 \, \sup_{x\in\Omega} \|(\cdot)_i\cdot U(x-\cdot) \|^2_{\mathcal{D}(A^*)}.
  \end{align*} 
   Let us set $f^x(\tilde{x},v):=v_i U(x-\tilde{x})$ for $(\tilde x,v)\in\R^{2d}$, and observe that from \eqref{eq:adj_A}
      \begin{align*}
   g^x:=A^*f^x = \Delta_v f^x - v\cdot \nabla_v f^x + v\cdot \nabla_{\tilde{x}} f^x = -v_i U(x-\tilde{x}) - v_i v\cdot \nabla_{\tilde{x}}U (x-\tilde{x}).
  \end{align*}
If $U \in H^1(\mathbb R^d)$, then $\{f^x\}_{x\in \Omega}$ and $\{g^x\}_{x\in \Omega}$ are uniformly bounded in $Y$. Together with the uniform boundedness of $\{\dot y_n\}$  in $L^2(0,T;Y_{-1})$ in $n$, we obtain that $\|\dot{\mathfrak{v}}_{n}\|_{L^2(0,T;(L^2(\Omega))^d)}$ is uniformly bounded in $n$.
  
  Next, again differentiation of convolutions allows to repeat the computations to obtain uniform boundedness of $\dot{\mathfrak{v}}_{n}$ in $L^2(0,T;(H^s(\Omega))^d)$
   provided that $U\in H^{s+1}(\mathbb R^d)$.
  
 Let us now argue that $\mathfrak{v}_{n,i}\to \mathfrak{v}_i $ in $L^2(0,T;C(\bar{\Omega}))$. By \eqref{eq:derive_y_aux5a} and Young's inequality for convolutions it can be argued that $y\to \mathfrak{v}_i=U*\rho_{\mu yv_i}$ is an element of  $\mathcal{L}(L^2(0,T;Y),L^2(0,T;(L^2(\mathbb R^d))))$. Since $y_n\rightharpoonup y$ in $L^2(0,T;Y)$ this implies that $\mathfrak{v}_{n,i} \rightharpoonup \mathfrak{v}_i$.

  Now let us choose $s>\tfrac{d}{2}$ and utilize that $U\in W^{s+1,1}(\mathbb R^d)\cap H^s(\mathbb R^d)$. Then due to the compact embedding of $L^2(0,T;H^{s+1}(\Omega))\cap W^{1,2}(0,T;H^{s-1}(\Omega)) $ in $L^2(0,T;H^s(\Omega))$, we have that $\mathfrak{v}_{n,i}\to \mathfrak{v}_i$ in $L^2(0,T;H^s(\Omega))$, by the Aubin-Lions lemma,   and consequently in $L^2(0,T;C(\bar{\Omega}))$ since $s>\tfrac{d}{2}$. Since this holds for each $i\in\{1,\dots,d\}$, we have $\mathfrak{v}_{n}\to \mathfrak{v}$ in $L^2(0,T;(H^s(\Omega))^d).$
  
  The convergence of $\mathfrak{u}_n \to \mathfrak{u}$ can be shown with almost identical arguments, replacing $f$ by $\tilde{f}(\tilde{x},v)=1_v U(x-\tilde{x})$ where $1_v(v)=1$ for all $v\in \mathbb R^d$ and using \eqref{eq:aux1} instead of \eqref{eq:derive_y_aux5}. 
 \end{proof}
\bibliographystyle{siam}
\bibliography{references}

\begin{thebibliography}{10}

\bibitem{AAMN24}
{\sc F.~Achleitner, A.~Arnold, V.~Mehrmann, and E.~A. Nigsch}, {\em
  Hypocoercivity in {H}ilbert spaces}, tech. rep., 2024.

\bibitem{AAS15}
{\sc F.~Achleitner, A.~Arnold, and D.~St{\"u}rzer}, {\em Large-time behavior in
  non-symmetric {F}okker-{P}lanck equations}, Rivista di Matematica della
  Universit\`{a} di Parma, 6 (2015), pp.~1--68.

\bibitem{Ada75}
{\sc R.~Adams}, {\em Sobolev Spaces}, New York Academic Press, 1975.

\bibitem{Albetal17}
{\sc G.~Albi, Y.-P. Choi, M.~Fornasier, and D.~Kalise}, {\em Mean field control
  hierarchy}, Applied Mathematics \& Optimization, 76 (2017), pp.~93--135.

\bibitem{Albetal24}
{\sc D.~Albritton, S.~Armstrong, J.-C. Mourrat, and M.~Novack}, {\em
  Variational methods for the kinetic {F}okker-{P}lanck equation}, Analysis \&
  PDE, 17 (2024), pp.~1953--2010.

\bibitem{BakGL14}
{\sc D.~Bakry, I.~Gentil, and M.~Ledoux}, {\em Analysis and Geometry of Markov
  Diffusion Operators}, Springer Heidelberg, 2014.

\bibitem{BelDT17}
{\sc N.~Bellomo, P.~Degond, and E.~Tadmor}, eds., {\em Active Particles, Volume
  1 - Advances in Theory, Models and Applications}, Birkhäuser, 2017.

\bibitem{BreK23}
{\sc T.~Breiten and K.~Kunisch}, {\em Improving the convergence rates for the
  kinetic {F}okker-{P}lanck equation by optimal control}, SIAM Journal on
  Control and Optimization, 61 (2023), pp.~1557--1581.

\bibitem{CarS95}
{\sc J.~A. Carillo and J.~Soler}, {\em On the initial value problem for the
  {V}lasov-{P}oisson-{F}okker-{P}lanck system with initial data in {$L^p$}
  spaces}, Mathematical Methods in the Applied Sciences, 18 (1995),
  pp.~825--839.

\bibitem{Cho16}
{\sc Y.~Choi}, {\em Global classical solutions of the
  {V}lasov–{F}okker–{P}lanck equation with local alignment forces},
  Nonlinearity, 29 (2016).

\bibitem{ChoHY23}
{\sc Y.-P. Choi, B.-H. Hwang, and Y.~Yoo}, {\em Global existence of weak
  solutions to the nonlinear {V}lasov-{F}okker-{P}lanck equation}, tech. rep.,
  2023.

\bibitem{ConG08}
{\sc F.~Conrad and M.~Grothaus}, {\em Construction of {$N$}-particle {L}angevin
  dynamics for {$H^{1,\infty}$}-potentials via generalized {D}irichlet forms},
  Potential Analysis, 28 (2008), pp.~261--282.

\bibitem{ConG10}
\leavevmode\vrule height 2pt depth -1.6pt width 23pt, {\em Construction,
  ergodicity and rate of convergence of {$N$}-particle {L}angevin dynamics with
  singular potentials}, Journal of Evolution Equations, 10 (2010),
  pp.~623--662.

\bibitem{CucS07}
{\sc F.~Cucker and S.~Smale}, {\em Emergent behavior in flocks}, IEEE
  Transactions on Automatic Control, 52 (2007), pp.~852--862.

\bibitem{CurZ95}
{\sc R.~Curtain and H.~Zwart}, {\em An Introduction to Infinite-Dimensional
  Linear Systems Theory}, Springer-Verlag, Berlin/Heidelberg, Germany, 1995.

\bibitem{DiLi89}
{\sc R.~DiPerna and P.~Lions}, {\em Global weak solutions of the
  {V}lasov-{M}axwell systems}, Communications in Pure and Applied Mathematics,
  17 (1989), pp.~729--757.

\bibitem{DuaFT10}
{\sc R.~Duan, M.~Fornasier, and G.~Toscani}, {\em A kinetic flocking model with
  diffusion}, Communications in Mathematical Physics, 300 (2010), pp.~95--145.

\bibitem{EngN99}
{\sc K.-J. Engel and R.~Nagel}, {\em One-Parameter Semigroups for Linear
  Evolution Equations}, Springer-Verlag, Berlin/Heidelberg, Germany, 1999.

\bibitem{GroS16}
{\sc M.~Grothaus and P.~Stilgenbauer}, {\em Hilbert space hypocoercivity for
  the {L}angevin dynamics revisited}, Methods of Functional Analysis and
  Topology, 22 (2016), pp.~152--168.

\bibitem{HaT08}
{\sc S.-Y. Ha and E.~Tadmor}, {\em From particle to kinetic and hydrodynamic
  descriptions of flocking}, Kinetic and Related Models, 1 (2008),
  pp.~415--435.

\bibitem{HelN05}
{\sc B.~Helffer and F.~Nier}, {\em Hypoelliptic Estimates and Spectral Theory
  for Fokker-Planck Operators and Witten Laplacians}, Springer, Berlin
  Heidelberg, 2005.

\bibitem{HosJS18}
{\sc R.~Hosfeld, B.~Jacob, and F.~Schwenninger}, {\em Integral input-to-state
  stability of unbounded bilinear control systems}, Mathematics of Control,
  Signals, and Systems,  (2022).

\bibitem{KarMT13}
{\sc T.~K. Karper, A.~Mellet, and K.~Trivisa}, {\em Existence of weak solutions
  to kinetic flocking models}, SIAM Journal on Mathematical Analysis, 45
  (2013), pp.~215--243.

\bibitem{Kat80}
{\sc T.~Kato}, {\em Perturbation Theory for Linear Operators}, Springer-Verlag,
  Berlin/Heidelberg, Germany, 1980.

\bibitem{LelS16}
{\sc T.~Leli\`{e}vre and G.~Stoltz}, {\em Partial differential equations and
  stochastic methods in molecular dynamics}, Acta Numerica,  (2016),
  pp.~681--880.

\bibitem{PicRT15}
{\sc B.~Piccoli, F.~Rossi, and E.~Tr\'{e}lat}, {\em Control to flocking of the
  kinetic {C}ucker--{S}male model}, SIAM Journal on Mathematical Analysis, 47
  (2015), pp.~4685--4719.

\bibitem{TucW09}
{\sc M.~Tucsnak and G.~Weiss}, {\em Observation and Control for Operator
  Semigroups}, Birkh\"auser Basel, Basel, 2009.

\bibitem{Vil09}
{\sc C.~Villani}, {\em Hypocoercivity}, Memoirs of the American Mathematical
  Society, 202 (2009), pp.~iv+141.

\end{thebibliography}

\end{document}